\documentclass[11pt,reqno]{amsart}
\usepackage{amsmath,amsxtra,latexsym,amsthm,amssymb,amscd,pb-diagram}
\usepackage{mathrsfs,mathabx,amsfonts}
\usepackage[colorlinks=true,linkcolor=magenta,citecolor=blue]{hyperref}
\usepackage[margin=1in]{geometry}

\usepackage{bm}
\usepackage{color}
\usepackage{makecell,multirow,diagbox}
\usepackage{mathtools}
\usepackage[displaymath, mathlines]{lineno}

\usepackage{graphicx}
\usepackage{epsfig}
\usepackage{array}
\usepackage{enumitem}

\newtheorem{theorem}{Theorem}[section]
\newtheorem{proposition}{Proposition}[section]
\newtheorem{lemma}{Lemma}[section]
\newtheorem{corollary}{Corollary}[section]
\newtheorem{remark}{Remark}[section]

\newcommand{\R}{\mathbb{R}}

\newcommand{\ity}{\infty}

\newcommand{\f}{\displaystyle\frac}

\begin{document}
\title[Semi-linear $\sigma$-evolution equations with double damping]{Applications of $L^p-L^q$ estimates for solutions to semi-linear $\sigma$-evolution equations with general double damping}

\thanks{$^* $\textit{Corresponding author:} Tuan Anh Dao (anh.daotuan@hust.edu.vn)}

\maketitle
\centerline{\scshape Dinh Van Duong$^1$, Tuan Anh Dao$^{1,*}$}
\medskip
{\footnotesize
	\centerline{$^1$ Faculty of Mathematics and Informatics, Hanoi University of Science and Technology}
	\centerline{No.1 Dai Co Viet road, Hanoi, Vietnam}}

\begin{abstract}
In this paper, we would like to study the linear Cauchy problems for semi-linear $\sigma$-evolution models with mixing a parabolic like damping term corresponding to $\sigma_1 \in [0,\sigma/2)$ and a $\sigma$-evolution like damping corresponding to $\sigma_2 \in (\sigma/2,\sigma]$. The main goals are on the one hand to conclude some estimates for solutions and their derivatives in $L^q$ setting, with any $q\in[1,\ity]$, by developing the theory of modified Bessel functions effectively to control oscillating integrals appearing the solution representation formula in a competition between these two kinds of damping. On the other hand, we are going to prove the global (in time) existence of small data Sobolev solutions in the treatment of the corresponding semi-linear equations by applying $(L^{m}\cap L^{q})- L^{q}$ and $L^{q}- L^{q}$ estimates, with $q\in (1,\ity)$ and $m\in [1,q)$, from the linear models. Finally, some further generalizations will be discussed in the end of this paper. 
\end{abstract}

% \linenumbers
\tableofcontents

%====================================================================================
%=================================================================================={Introduction}	
\section{Introduction}
In this paper, let us consider the following Cauchy problem for semi-linear $\sigma$-evolution equations with double damping:
\begin{equation} \label{Main.Eq.1}
\begin{cases}
u_{tt}+ (-\Delta)^\sigma u+ \mu_1(-\Delta)^{\sigma_1} u_t+ \mu_2(-\Delta)^{\sigma_2} u_t= |u|^p, &\quad x\in \R^n,\, t \ge 0, \\
u(0,x)= u_0(x),\quad u_t(0,x)= u_1(x), &\quad x\in \R^n, \\
\end{cases}
\end{equation}
where $\sigma\ge 1$, $\sigma_1$, $\sigma_2$ are any constants being subject to $0< \sigma_1< \sigma/2< \sigma_2< \sigma$. Here $\mu_1,\mu_2 $ are positive constants and the parameter $p>1$ stands for power exponents of the nonlinear term. The homogeneous problem corresponding to \eqref{Main.Eq.1} is
\begin{align}
    \begin{cases}
         u_{tt}+ (-\Delta)^\sigma u+ \mu_1(-\Delta)^{\sigma_1} u_t+ \mu_2(-\Delta)^{\sigma_2} u_t = 0, &\quad x \in \R^n,\, t \ge 0,\\
       u(0,x)= u_0(x),\quad u_t(0,x)= u_1(x), &\quad x\in \R^n.  
    \end{cases}\label{Main.Eq.3}
\end{align}

Recently there are a number of papers (for example \cite{DaoReissig1}, \cite{DaoReissig2}) presenting the case $\mu_2 = 0$, that is, parabolic like semi-linear structurally damped $\sigma$-evolution models
\begin{align} \label{Eq4}
    \begin{cases}
        u_{tt} + (-\Delta)^{\sigma}u + \mu_1 (-\Delta)^{\sigma_1}u_t = F(u, u_t), &\quad x \in \R^n,\, t \ge 0,\\
        u(0,x) = u_0(x), \quad u_t(0,x) = u_1(x), &\quad x\in \R^n,
      \end{cases}
    \end{align}
with $\sigma_1 \in (0, \frac{\sigma}{2})$ and $F(u,u_t) = |u|^p$ or $|u_t|^p$. In \cite{DaoReissig1}, the authors divided the phase space into two parts: frequencies small enough and large enough to study the Fourier multiplier with fluctuations in representing the solution of the homogeneous Cauchy problem corresponds to (\ref{Eq4}). They applied the modified Bessel function theory combined with Fa\`{a} di Bruno's formula in the relations with the Fourier coefficients appearing in wave models, corresponding to the small frequencies and high frequencies. Therefore, having estimates $L^1$ for the oscillating integral is to conclude that the estimate $L^p- L^q$ does not necessarily lie on the conjugate for the solutions of the homogeneous Cauchy problem. Furthermore, the critical exponent for the existence of a global solution of the problem (\ref{Eq4}) has also been studied. Using the same tools as in \cite{DaoReissig1}, the authors of \cite{DaoReissig2} also found the $L^p-L^q$ estimate of the homogeneous Cauchy problem (\ref{Main.Eq.3}) and proved the result that a complete solution exists locally (over time) with the small data of problem (\ref{Main.Eq.1}) in the case $\mu_1 = 0$, that is, semi-linear models with $\sigma$-evolutionlike structural damping
\begin{align*}
     \begin{cases}
        u_{tt} + (-\Delta)^{\sigma}u + \mu_2 (-\Delta)^{\sigma_2}u_t = F(u, u_t), &\quad x \in \R^n,\, t \ge 0,\\
        u(0,x) = u_0(x), \quad u_t(0,x) = u_1(x), &\quad x\in \R^n,
      \end{cases}
\end{align*}
where $\sigma_2 \in (\frac{\sigma}{2}, \sigma)$ and $F(u,u_t) = |u|^p$ or $|u_t|^p$.

In \cite{DabbiccoEbert2017}, the authors have proven the blow up and lifespan results, which means that the nonlinear problem (\ref{Eq4}) will not has a global solution over time when the exponent $p$ has a value smaller than the critical value. Next for the case $\sigma_1 = 0, \sigma_2=\sigma =1$, that is, the wave equations with frictional and visco-elastic damping terms 
\begin{align*}
    \begin{cases}
        u_{tt} -\Delta u +u_t -\Delta u_t = |u|^p, &\quad x \in \R^n,\, t \ge 0,\\
        u(0,x) = u_0(x), \quad u_t(0,x) = u_1(x), &\quad x\in \R^n,
    \end{cases}
\end{align*}
people prove the global (in
time) existence of small data energy solutions and analyze the large time
behavior of these global solutions as well by mixing additional $L^m$ regularity for the data (see \cite{DaoMichihisa2020}). Moreover, issues related to the critical exponent of this problem have been studied in \cite{IkehataTakeda2017} and the authors in \cite{IkehataSawada2016} have described the asymptotic behavior of the problem in case the problem has global experience. In particular, in the case $\sigma = 1$, that is, the semilinear waves equations with two dissipative terms
\begin{align*}
    \begin{cases}
        u_{tt} -\Delta u +(-\Delta)^{\sigma_1} u_t + (-\Delta)^{\sigma_2} u_t = F(u,u_t), &\quad x \in \R^n,\, t \ge 0,\\
        u(0,x) = u_0(x),\quad u_t(0,x) = u_1(x), &\quad x\in \R^n,
    \end{cases}
\end{align*}
where  $0 \leq \sigma_1 < 1/2 < \sigma_2 \leq 1$ and $F(u,u_t) = |u|^p$ or $|u_t|^p$ , the authors in \cite{ChenDAbbiccoGirardi2022} used the Mikhlin–Hörmander multiplier
theorem and Hardy–Littlewood theorem for the Riesz potential to obtain $L^r-L^q$ estimates with $1 < r \leq q < \infty$ for the homogeneous Cauchy problem 
\begin{equation*}
    \begin{cases}
        u_{tt}-\Delta u +(-\Delta)^{\sigma_1}u_t + (-\Delta)^{\sigma_2} u_t = 0, &\quad x \in \R^n,\, t \ge 0,\\
        u(0,x) = u_0(x),\quad u_t(0,x) = u_1(x), &\quad x\in \R^n.
    \end{cases}
\end{equation*}
Then, they prove  global (in time) existence results and the blow-up results for the above problem.

The goal of this paper is to find an estimate of $L^q-L^q$ for the homogeneous problem (\ref{Main.Eq.3}) with $1 \leq q \leq \infty$ and by adding $L^m$ regularity to the data, we prove the existence of complete locality (over time) of the solution to nonlinear problems (\ref{Main.Eq.1}) with small data and make some generalizations for (\ref{Main.Eq.1}). To obtain an estimate of $L^p-L^q$ for the solution of the homogeneous problem (\ref{Main.Eq.3}), we need to obtain an estimate of $L^1$ for the oscillatory integrals using the theorems for the modified Bessel function. The obtained results show that the solution loses its decay for small times. To overcome this, we increased the regularity of the data to obtain a solution estimate that reduces to $0$ when $t \rightarrow 0^{+}$.

\textbf{This paper is organized as follows:} In Section \ref{Linear estimates} we find the $(L^m \cap L^q)-L^q$ estimates and $L^q-L^q$ estimates for the solution of the homogeneous Cauchy problem (\ref{Main.Eq.3}), where $1 \leq m < q \leq \infty$. Then, applying the derived estimates
and Banach’s fixed point theorem, we prove global (in time) existence results for the problem (\ref{Main.Eq.1}). Next, we prove the extension result of the problem (\ref{Main.Eq.1}) by proving the existence of a unique global solution over time with small data of the problem (\ref{Main.Eq.2}) in Section \ref{Sec.3}. Finally, Section \ref{Sec.4} is to devote to a investigation of several further
generalizations.\medskip
%...............................

\textbf{Notations}
\begin{itemize}[leftmargin=*]
\item We write $f\lesssim g$ when there exists a constant $C>0$ such that $f\le Cg$, and $f \sim g$ when $g\lesssim f\lesssim g$.
\item We denote $\widehat{w}(t,\xi):= \mathfrak{F}_{x\rightarrow \xi}\big(w(t,x)\big)$ as the Fourier transform with respect to the spatial variable of a function $w(t,x)$. Moreover, $\mathfrak{F}^{-1}$ represents the inverse Fourier transform.
\item As usual, $H^{a}_q$ and $\dot{H}^{a}_q$, with $a \ge 0$, denote Bessel and Riesz potential spaces based on $L^q$ spaces with $1 \leq q \leq \infty$. Here $\big<D\big>^{a}$ and $|D|^{a}$ stand for the pseudo-differential operators with symbols $\big<\xi\big>^{a}$ and $|\xi|^{a}$, respectively.
\item For any $\gamma \in \R$, we denote $[\gamma]^+:= \max\{\gamma,0\}$ as its positive part and $\lceil \gamma \rceil := \min\{k \in \mathbb{Z}: k \geq \gamma\}$.  
\item Finally, we introduce the space
$\frak{D}^{s}_{m,q}:= \big(L^m \cap H^s_q\big) \times \big(L^m\cap H^{[s-\sigma+\sigma_2]^{+}}_q\big)$ with the norm
$$\|(\phi_0,\phi_1)\|_{\frak{D}_{m,q}^{s}}:=\|\phi_0\|_{H^s_q}+ \|\phi_0\|_{L^m}+ \|\phi_1\|_{H^{[s-\sigma+\sigma_2]^{+}}_q}+ \|\phi_1\|_{L^m} \quad \text{ with } s \geq 0. $$
\end{itemize}
\medskip

\textbf{Main results}
\begin{theorem}\label{theorem1.0}
    Let $q \in (1, \infty)$ be a fixed constant, $m \in [1, q)$ and $0 \leq s < 2\sigma_2$. We assume that the exponent $p$ satisfies the conditions $p > 1+ \lceil [s-\sigma+\sigma_2]^{+} \rceil$ and
    \begin{equation}\label{condition1.0.1}
        p > 1 +\frac{4m(\sigma-\sigma_1)}{n-2m(\sigma-\sigma_1)} \text{ and } n > 2m(\sigma-\sigma_1).
    \end{equation}
    Moreover, we suppose the following conditions:
\begin{equation}\label{condition1.0.2}
    p \in \left[\frac{q}{m}, \infty\right) \text{ if } n \leq qs \text{ or } p \in \left[\frac{q}{m}, \frac{n -q[s-\sigma+\sigma_2]^{+}}{n-qs}\right] \text{ if } n \in \left(qs, \frac{q^2s-qm[s-\sigma+\sigma_2]^{+}}{q-m}\right].
\end{equation}
Then, there exists a constant $\epsilon > 0$ such that for any small data
\begin{equation*}
    (u_0, u_1) \in \frak{D}_{m,q}^s \text{ satisfying the assumption } \|(u_0,u_1)\|_{\frak{D}_{m,q}^s} \leq \epsilon,
\end{equation*}
we have a uniquely determined global (in time) small data energy solution
\begin{equation*}
    u \in C\left([0, \infty\right), H^{s}_q)
\end{equation*}
to (\ref{Main.Eq.1}). The following estimates hold:
\begin{align*}
    \|u(t,\cdot)\|_{L^q} &\lesssim (1+t)^{1-\frac{n}{2(\sigma-\sigma_1)}(\frac{1}{m}-\frac{1}{q})} \|(u_0,u_1)\|_{\frak{D_{m,q}^{s}}},\\
    \||D|^{s}  u(t,\cdot)\|_{L^q} &\lesssim (1+t)^{1-\frac{n}{2(\sigma-\sigma_1)}(\frac{1}{m}-\frac{1}{q})-\frac{s}{2(\sigma-\sigma_1)}} \|(u_0,u_1)\|_{\frak{D}_{m,q}^s}.
\end{align*}
\end{theorem}

\begin{theorem}\label{theorem1.2}
     Let $q \in (1, \infty)$ be a fixed constant, $m \in [1, q)$. We assume that the exponent $p$ satisfies the conditions $p > 1+ \lceil 3\sigma_2-\sigma \rceil$ and
    \begin{equation}\label{condition1.2.1}
        p > 1 +\frac{4m(\sigma-\sigma_1)}{n-2m(\sigma-\sigma_1)} \text{ and } n > 2m(\sigma-\sigma_1).
    \end{equation}
    Moreover, we suppose the following conditions:
\begin{equation}\label{condition1.2.2}
    p \in \left[\frac{q}{m}, \infty\right) \text{ if } n \leq 2q\sigma_2 \text{ or } p \in \left[\frac{q}{m}, \frac{n-q(3\sigma_2-\sigma)}{n-2q\sigma_2}\right] \text{ if } n \in \left(2q\sigma_2, \frac{2q^2\sigma_2-qm(3\sigma_2-\sigma)}{q-m}\right].
\end{equation}
Then, there exists a constant $\epsilon > 0$ such that for any small data
\begin{equation*}
    (u_0, u_1) \in \frak{D}_{m,q}^{2\sigma_2} \text{ satisfying the assumption } \|(u_0,u_1)\|_{\frak{D}_{m,q}^{2\sigma_2}} \leq \epsilon,
\end{equation*}
we have a uniquely determined global (in time) small data energy solution
\begin{equation*}
    u \in C\left([0, \infty\right), H^{2\sigma_2}_q) \cap C^1\left([0, \infty\right), L^q)
\end{equation*}
to (\ref{Main.Eq.1}). The following estimates hold:
\begin{align*}
    \|u(t,\cdot)\|_{L^q} &\lesssim (1+t)^{1-\frac{n}{2(\sigma-\sigma_1)}(\frac{1}{m}-\frac{1}{q})} \|(u_0,u_1)\|_{\frak{D}_{m,q}^{2\sigma_2}},\\
    \||D|^{2\sigma_2}  u(t,\cdot)\|_{L^q} &\lesssim (1+t)^{1-\frac{n}{2(\sigma-\sigma_1)}(\frac{1}{m}-\frac{1}{q})-\frac{\sigma_2}{\sigma-\sigma_1}} \|(u_0,u_1)\|_{\frak{D}_{m,q}^{2\sigma_2}},\label{estimates1.2.2}\\
    \|u_t(t,\cdot)\|_{L^q} &\lesssim (1+t)^{1-\frac{n}{2(\sigma-\sigma_1)}(\frac{1}{m}-\frac{1}{q})-\frac{\sigma_1}{\sigma-\sigma_1}} \|(u_0,u_1)\|_{\frak{D}_{m,q}^{2\sigma_2}}.
\end{align*}
\end{theorem}
\begin{theorem}\label{theorem1.3}
     Let $q \in (1, \infty)$ be a fixed constant, $m \in [1, q)$ and $s> 2\sigma_2$. We assume that the exponent $p$ satisfies the conditions $p > 1 +\lceil s -\sigma +\sigma_2 \rceil$ and 
    \begin{equation}\label{condition1.3.1}
        p > 1 +\frac{\max\{ms, 4m(\sigma-\sigma_1)\}}{n-2m(\sigma-\sigma_1)} \text{ and } n > 2m(\sigma-\sigma_1).
    \end{equation}
    Moreover, we suppose the following conditions:
\begin{equation}\label{condition1.3.2}
    p \in \left[\frac{q}{m}, \infty\right) \text{ if } n \leq qs \text{ or }  p \in \left[\frac{q}{m}, \frac{n -q(s-\sigma+\sigma_2)}{n-qs}\right] \text{ if } n \in \left(qs, \frac{q^2s-qm(s-\sigma+\sigma_2)}{q-m}\right].
\end{equation}
Then, there exists a constant $\epsilon > 0$ such that for any small data
\begin{equation*}
    (u_0, u_1) \in \frak{D}_{m,q}^s \text{ satisfying the assumption } \|(u_0,u_1)\|_{\frak{D}_{m,q}^s} \leq \epsilon,
\end{equation*}
we have a uniquely determined global (in time) small data energy solution
\begin{equation*}
    u \in C\left([0, \infty\right), H^{s}_q) \cap C^1\left([0, \infty\right), H^{s-2\sigma_2}_q)
\end{equation*}
to (\ref{Main.Eq.1}). The following estimates hold:
\begin{align*}
    \|u(t,\cdot)\|_{L^q} &\lesssim (1+t)^{1-\frac{n}{2(\sigma-\sigma_1)}(\frac{1}{m}-\frac{1}{q})} \|(u_0,u_1)\|_{\frak{D}_{m,q}^s},\\
    \||D|^{s}  u(t,\cdot)\|_{L^q} &\lesssim (1+t)^{1-\frac{n}{2(\sigma-\sigma_1)}(\frac{1}{m}-\frac{1}{q})-\frac{s}{2(\sigma-\sigma_1)}} \|(u_0,u_1)\|_{\frak{D}_{m,q}^s},\label{estimates1.3.2}\\
    \|u_t(t,\cdot)\|_{L^q} &\lesssim (1+t)^{1-\frac{n}{2(\sigma-\sigma_1)}(\frac{1}{m}-\frac{1}{q})-\frac{\sigma_1}{\sigma-\sigma_1}} \|(u_0,u_1)\|_{\frak{D}_{m,q}^s},\\
    \||D|^{s-2\sigma_2} u_t(t,\cdot)\|_{L^q} &\lesssim (1+t)^{1-\frac{n}{2(\sigma-\sigma_1)}(\frac{1}{m}-\frac{1}{q})-\frac{s}{2(\sigma-\sigma_1)}+\frac{\sigma_2-\sigma_1}{\sigma-\sigma_1}} \|(u_0,u_1)\|_{\frak{D}_{m,q}^s},
\end{align*}
\end{theorem}

\section{Estimates for solutions to the linear Cauchy problem} \label{Linear estimates}
\noindent In this section, we will consider the homogeneous Cauchy problem (\ref{Main.Eq.3}). Applying the Fourier transform with $\hat{u}(t,\xi)=\mathfrak{F}_{x\to\xi}(u(t,x))$, we have
\begin{equation}\label{5.83}
\begin{cases}
    \widehat{u}_{tt}(t,\xi) + (\mu_1|\xi|^{2\sigma_1}+\mu_2|\xi|^{2\sigma_2})\widehat{u}_t(t,\xi) + |\xi|^{2\sigma} \widehat{u}(t,\xi) = 0  \\
    \widehat{u}(0,\xi) = \widehat{u}_0(\xi),\quad \widehat{u}_t(0,\xi)= \hat{u}_1(\xi).
\end{cases}
\end{equation}
The solution to $(\ref{5.83})$ is\\
\begin{equation*}
    \widehat{u}(t,\xi) = \widehat{K_0}(t,\xi) \widehat{u_0}(\xi) + \widehat{K_1}(t,\xi) \widehat{u_1}(\xi),
\end{equation*}
so that we can write the solution to $(\ref{Main.Eq.3})$ by the formula
\begin{equation*}
u(t,x) = K_0(t,x) * u_0(x) + K_1(t,x) * u_1(x),
\end{equation*}
where 
\begin{equation*}
\widehat{K_0}(t,\xi) = \frac{\lambda_{1} e^{\lambda_{2}t}-\lambda_{2}e^{\lambda_{1}t}}{\lambda_{1} - \lambda_{2}}
\text{ and }
\widehat{K_1}(t,\xi) = \frac{e^{\lambda_{1}t}-e^{\lambda_{2}t}}{\lambda_{1} - \lambda_{2}}.
\end{equation*}
The characteristic roots $\lambda_{ \pm} =\lambda_{ \pm}(|\xi|)$ are given by
\begin{equation*}
\lambda_{1,2}(|\xi|) =
\begin{cases}
    \displaystyle\f{1}{2}\left(-\mu_1|\xi|^{2\sigma_1} - \mu_2|\xi|^{2\sigma_2} \pm \sqrt{(\mu_1|\xi|^{2\sigma_1} + \mu_2|\xi|^{2\sigma_2})^{2} - 4 |\xi|^{2\sigma}}\right) & \text { if } \xi \in \mathbb{R}^n \setminus \Omega, \\
     \displaystyle\f{1}{2}\left(-\mu_1|\xi|^{2\sigma_1} - \mu_2|\xi|^{2\sigma_2} \pm i \sqrt{4 |\xi|^{2\sigma}-(\mu_1|\xi|^{2\sigma_1} + \mu_2|\xi|^{2\sigma_2})^{2} }\right) & \text { if } \xi \in \Omega,  
\end{cases}
\end{equation*}
where $\Omega = \left\{\xi \in \mathbb{R}^n: \mu_1|\xi|^{2\sigma_1} + \mu_2|\xi|^{2\sigma_2} < 2 |\xi|^{\sigma} \right\}$. In this paper, without loss of generality, we only need to consider the case $\mu_1=\mu_2 = 1$. Because of $0 < \sigma_1 <\sigma/2 < \sigma_2 < \sigma$, there exists a sufficiently small constant $\varepsilon^*>0$ such that
\begin{equation*}
    (-\ity, \varepsilon^*)\cup \left(\frac{1}{\varepsilon^*},\ity\right) \subset \Omega.
\end{equation*}
Then, taking account of the cases of small and large frequencies separately, one sees
\begin{align}
     &\lambda_{1} \sim -|\xi|^{2(\sigma-\sigma_1)},\quad 
    \lambda_{2} \sim -|\xi|^{2\sigma_1}, \quad \lambda_1-\lambda_2 \sim |\xi|^{2\sigma_1}\quad \text{ for } |\xi| \leq \varepsilon^*, \label{approx1} \\
    &\lambda_{1} \sim -|\xi|^{2(\sigma-\sigma_2)} ,\quad \lambda_{2} \sim -|\xi|^{2\sigma_2}, \quad \lambda_1-\lambda_2 \sim |\xi|^{2\sigma_2} \quad \text{ for }|\xi| \geq \frac{1}{\varepsilon^*}. \label{approx2}
\end{align} 
Let $\chi_k= \chi_k(r)$ with $k\in\{\rm L,M,H\}$ be smooth cut-off functions having the following properties:
\begin{align*}
&\chi_{\rm L}(r)=
\begin{cases}
1 &\quad \text{ if }r\le \varepsilon^*/2, \\
0 &\quad \text{ if }r\ge \varepsilon^*,
\end{cases}
\qquad
\chi_{\rm H}(r)=
\begin{cases}
1 &\quad \text{ if }r\ge 2/\varepsilon^*, \\
0 &\quad \text{ if }r\le 1/\varepsilon^*,
\end{cases} \\ 
&\text{and } \chi_{\rm M}(r)= 1- \chi_{\rm L}(r)- \chi_{\rm H}(r).
\end{align*}
We note that $\chi_{\rm M}(r)= 1$ if $\varepsilon^* \le r\le 1/\varepsilon^*$ and $\chi_2(r)= 0$ if $r \le \varepsilon^*/2$ or $r \ge 2/\varepsilon^*$. We now decompose the solution (\ref{Main.Eq.3}) into three parts localized separately to low, middle and high frequencies, that is,
\begin{equation*}
u(t,x)= u_{\chi_{\rm L}}(t,x)+ u_{\chi_{\rm M}}(t,x)+ u_{\chi_{\rm H}}(t,x),
\end{equation*}
where
$$u_{\chi_k}(t,x)= \mathfrak{F}^{-1}\big(\chi_k(|\xi|)\widehat{u}(t,\xi)\big)\quad \text{ with } k= \text{L, M, H}. $$
Here $\mathfrak{F}^{-1}$ stands for the inverse Fourier transform.
\subsection{$L^1$ estimates}

\begin{lemma}\label{Lemma2.1}
    The following estimates hold:
    \begin{align}
        \left\|\frak{F}^{-1}\left(|\xi|^{s}e^{\lambda_1(|\xi|) t}\chi_{\rm H}(|\xi|)\right)(t,\cdot)\right\|_{L^1} \lesssim 
    \begin{cases}
        t^{-\frac{s}{2(\sigma-\sigma_2)}} &\text{ if } t \in (0,1],\\
        e^{-ct} &\text{ if } t \in [1, \infty),
    \end{cases}\label{estimate2.1.100}
    \end{align}
    where $c$ is suitable positive constant.
\end{lemma}
Before proving Lemma \ref{Lemma2.1}, we need to prove the following auxiliary proposition:
\begin{proposition}\label{proposition2.1}
    The following estimate hold:
    \begin{equation}\label{Es.of.Pro2.1}
        \left|\partial_\eta^j e^{\lambda_1(\eta)}\chi_{\rm H}(\eta)\right| \lesssim e^{-c\eta^{2(\sigma-\sigma_2)}} \eta^{2(\sigma-\sigma_2)-j} 
    \end{equation}
    for all $j \geq 1$ and $c$ is suitable positive constant.
\end{proposition}
\begin{proof}
   First, we rewrite $\lambda_1(\eta)$ as follows
    \begin{align*}
        \lambda_1(\eta) =& -\eta^{2(\sigma-\sigma_2)}\left(\frac{1}{\eta^{2(\sigma_1-\sigma_2)}+1}\frac{1}{\frac{1}{2}+\frac{1}{2}\sqrt{1-\frac{4\eta^{2(\sigma-2\sigma_2)}}{\left(\eta^{2(\sigma_1-\sigma_2)}+1\right)^2}}}\right)\\
        :=& u(\eta) v(\eta) 
    \end{align*}
    Applying Faa di Bruno's formula we obtain
    \begin{align}
        \partial_{\eta}^j e^{\lambda_1(\eta)t} = 
        \displaystyle \sum a_i e^{\lambda_1(\eta)}\displaystyle \prod_{i=1}^j \left(\lambda_1^{(i)}(\eta)\right)^{m_i}\label{equation.pro2.1.1}
    \end{align}
    where set of non-negative integers $(m_1, m_2,..., m_n)$ satisfying the constraint of the following Diophantine equation: $m_1 + 2m_2+...+jm_j =j$. Next, applying Leibniz's formula to the higher derivative we have
    \begin{align*}
        \lambda_1^{(i)}(\eta) =\left(u(\eta)v(\eta)\right)^{(i)} =& \displaystyle \sum_{k=0}^{i}b_k u^{(i-k)}(\eta) v^{(k)}(\eta)\\
        =&\displaystyle\sum_{k=0}^ib_k\eta^{2(\sigma-\sigma_2)-i+k}v^{(k)}(\eta).
    \end{align*}
    For this reason, we need to estimate $v^{(k)}(\eta)$. Use Maclaurin expansion for functions 
    \begin{align*}
    v_1(\eta) &= \frac{1}{\eta^{2(\sigma_1-\sigma_2)}+1},\\
    v_2(\eta) &= \frac{1}{\frac{1}{2}+\frac{1}{2}\sqrt{1-\frac{4\eta^{2(\sigma-2\sigma_2)}}{\left(\eta^{2(\sigma_1-\sigma_2)}+1\right)^2}}} = \frac{1}{\frac{1}{2}+\frac{1}{2}\sqrt{1-4\eta^{2(\sigma-2\sigma_2)}(v_1(\eta))^2}}
    \end{align*}
    we obtain
    \begin{align*}
        v(\eta) &= v_1(\eta) v_2(\eta)\\
        &= 1 + c_1 \eta^{2(\sigma-2\sigma_2)} +c_2 \eta^{2(\sigma_1-\sigma_2)} + A(\eta).
    \end{align*}
    where 
    \begin{equation*}
        A(\eta) = \displaystyle \sum_{\ell > \max\{2\sigma_2-\sigma, 2(\sigma_2-\sigma_1)\}} c_\ell\eta^{-\ell}. 
    \end{equation*}
From here, we have a conclusion
\begin{align*}
   | v^{(k)}(\eta)| \lesssim
   \eta^{-k} 
\end{align*}
for all $k \geq 0$.
This implies
\begin{align*}
    |\lambda_1^{(i)}(\eta)| \lesssim \eta^{2(\sigma-\sigma_2)-i}
\end{align*}
for all $i \geq 0$. Combined with (\ref{equation.pro2.1.1}), we have a conclusion for Proposition \ref{proposition2.1}.
\end{proof}
\begin{corollary}\label{corollary1}
    The following estimate hold:
    \begin{align}\label{Es.of.Co1}
        \left|\partial_r^j\left(e^{\lambda_1(rt^{-\frac{1}{2(\sigma-\sigma_2)}})t}\right) \chi_{\rm H}\left(rt^{-\frac{1}{2(\sigma-\sigma_2)}}\right)
\right| \lesssim e^{-cr^{2(\sigma-\sigma_2)}}r^{2(\sigma-\sigma_2)-j}
    \end{align}
    for all $j \geq 1$, $t \in (0,1]$ and c is suitable positive constant.
\end{corollary}
\begin{proof}
    We have a performance later
    \begin{align*}
        \lambda_1\left(rt^{-\frac{1}{2(\sigma-\sigma_2)}}\right)t = r^{2(\sigma-\sigma_2)}v\left(r t^{-\frac{1}{2(\sigma-\sigma_2)}}\right).
    \end{align*}
    Here the function $v(\eta)$ is defined the same as in the proof of Proposition \ref{proposition2.1}. Perform the same proof steps as in proof of Proposition \ref{proposition2.1}, noting that
    \begin{align*}
       \left( v\left(r t^{-\frac{1}{2(\sigma-\sigma_2)}}\right)\right)^{(k)} &= t^{-\frac{k}{2(\sigma-\sigma_2)}} v^{(k)}\left(r t^{-\frac{1}{2(\sigma-\sigma_2)}}\right)\\
       &\lesssim r^{-k}
    \end{align*}
    for all $k \geq 0$,
    we have a conclusion Corollary \ref{corollary1}.
\end{proof}
\begin{proof}[Proof of Lemma \ref{Lemma2.1}]
    \begin{itemize}[leftmargin=*]
    \item[$\bullet$] \textbf{We first prove the estimate for $t \geq 1$}. 
With $|x| \leq \frac{1}{\epsilon^{*}}$ we have
 \begin{align*}
    &\left\|\frak{F}^{-1}\left(e^{\lambda_1(|\xi|) t}|\xi|^{s}\chi_{\rm H}(|\xi|)\right)(\chi_{\rm L}(|x|)+\chi_{\rm M}(|x|))\right\|_{L^1}\\ &\qquad\lesssim \int_{\mathbb{R}^n} (\chi_{\rm L}(|x|)+\chi_{\rm M}(|x|)) dx \int_{\mathbb{R}^n}e^{-|\xi|^{2(\sigma-\sigma_2)}t}|\xi|^s \chi_{\rm H}(|\xi|) d\xi\\
    &\qquad\lesssim e^{-ct}.
\end{align*}
Next we will prove the above estimate for $|x| > \frac{1}{\epsilon^{*}}$. We introduce the function
    \begin{equation*}
        F(t,x):= \frak{F}^{-1}\left(e^{\lambda_1(|\xi|) t}|\xi|^{s}\chi_{\rm H}(|\xi|)\right)(t,x).
    \end{equation*}
    Using the modified Bessel functions we get
    \begin{align*}
        F(t,x):= c\int_0^{\infty} e^{\lambda_1(r)t}r^{s+n-1}\chi_{\rm H}(r) \tilde{\mathcal{J}}_{\frac{n}{2}-1}(r|x|)dr.
    \end{align*}
    \textit{Let us consider odd spatial dimensions $n = 2m+1, m \geq 0$.} 
    By introducing the vector field
    $$\mathcal{V}(f(r)) := \f{d}{dr}\left(\frac{1}{r}f(r)\right),$$
    we carry out $m+1$ steps of partial integration to obtain 
    \begin{align*}
        F(t,x) =& -\frac{c}{|x|^n} \int_0^{\infty}\partial_r\left(\mathcal{V}^m \left(e^{\lambda_1(r)t}r^{s+2m}\chi_{\rm H}(r)\right)\right)\sin(r|x|)dr\\
        =& \displaystyle \sum_{j=0}^m \displaystyle \sum_{k=0}^{j+1} \frac{c_{jk}}{|x|^n}\int_0^{\infty} \partial_r^{j+1-k} e^{\lambda_1(r) t} \partial_r^k \chi_{\rm H}(r) r^{s+j} \sin(r|x|)dr\\
        &+ \displaystyle \sum_{j=0}^m \displaystyle \sum_{k=0}^j \frac{c_{jk}}{|x|^n}\int_0^{\infty}\partial_r^{j-k}e^{\lambda_1(r)t}\partial_r^{k+1}\chi_{\rm H}(r) r^{s+j} \sin(r|x|) dx \\
        &+ \displaystyle \sum_{j=1}^m \displaystyle \sum_{k=0}^j \frac{c_{jk}}{|x|^n}\int_0^{\infty} \partial_r^{j-k}e^{\lambda_1(r)t}\partial_r^k \chi_{\rm H}(r)r^{s+j-1} \sin(r|x|)dr
    \end{align*}
    with some constants $c_{jk}$. 
    Now, we consider the integrals 
    \begin{equation*}
        F_{j,0} := \int_0^{\infty} \partial_r^{j+1} e^{\lambda_1(r) t} \chi_{\rm H}(r) r^{s+j} \sin(r|x|)dr,
    \end{equation*}
    where $j = 0,...,m$. We note that for $j = 0,.., m$ it holds
    \begin{align*}
        \left|\partial_r^{j+1} e^{\lambda_1(r) t}\right| \lesssim r^{2(\sigma-\sigma_2)-j-1} e^{-ct},
    \end{align*}
    on the support of $\chi_{\rm H}$ and on the support of its derivatives. Hence, the splitting of the integral $J_{j,0}$ gives on the one hand
    \begin{align*}
        \int_0^{\frac{1}{|x|}} \left|\partial_r^{j+1} e^{\lambda_1(r) t} \chi_{\rm H}(r) r^{s+j} \sin(r|x|)\right|dr &\lesssim e^{-ct} \int_0^{\frac{1}{|x|}} r^{2(\sigma-\sigma_2)+s-1} dr \\
        &\lesssim \frac{e^{-ct}}{|x|^{2(\sigma-\sigma_2)+s}},
    \end{align*}
and on the other hand
\begin{align*}
    \left|\int_{\frac{1}{|x|}}^{\infty} \partial_r^{j+1} e^{\lambda_1(r) t} \chi_{\rm H}(r) r^{s+j} \sin(r|x|)dr\right|
    \leq& \frac{1}{|x|} \left|\partial_r^{j+1} e^{\lambda_1(r) t} \chi_{\rm H}(r) r^{s+j} \cos(r|x|)\right|_{r = \frac{1}{|x|}}\\
    &+ \frac{1}{|x|}\left|\int_\frac{1}{|x|}^{\infty} \partial_r \left(\partial_r^{j+1} e^{\lambda_1(r)t}\chi_{\rm H}(r)r^{s+j} \right)\cos(r|x|) dr\right|\\
     \lesssim& \frac{e^{-ct}}{|x|^{2(\sigma-\sigma_2)+s}} + \frac{e^{-ct}}{|x|} \int_{\frac{1}{|x|}}^{\infty} e^{-c_1r^{2(\sigma-\sigma_2)}t} r^{2(\sigma-\sigma_2)+s-2} dr\\
     \lesssim& \frac{e^{-ct}}{|x|^{2(\sigma-\sigma_2)+s}}.
\end{align*}
Summarizing, we have shown
\begin{equation*}
    |F_{j,0}| \lesssim \frac{e^{-ct}}{|x|^{2(\sigma-\sigma_2)+s}}.
\end{equation*}
Taking $k \geq 1$ we study the integrals
    \begin{equation}\label{equation2.1.1}
        F_{j,k} := \int_0^{\infty} \partial_r^{j+1-k} e^{\lambda_1(r) t} \partial_r^k \chi_{\rm H}(r) r^{s+j} \sin(r|x|)dr.
    \end{equation}
The splitting of the integral $F_{j,k}$ gives on the one hand
\begin{align*}
    &\left|\int_0^{\frac{1}{|x|}}\partial_r^{j+1-k}e^{\lambda_1(r)t}\partial_r^k\chi_{\rm H}(r)r^{s+j}\sin(r|x|)dr\right| \\
    &\qquad\lesssim e^{-ct} \int_0^{\frac{1}{|x|}} r^{2(\sigma-\sigma_2)+k-1+s} \partial_r^k\chi_{\rm H}(r) dr \lesssim \frac{e^{-ct}}{|x|^{2(\sigma-\sigma_2)+k+s}},
\end{align*}
and on the other hand
\begin{align*}
    &\left|\int_{\frac{1}{|x|}}^{\infty} \partial_r^{j+1-k}e^{\lambda_1(r)t}\partial_r\chi_{\rm H}(r)r^{s+j}\sin(r|x|)dr\right| \\
    &\qquad\le \frac{1}{|x|}\left|\partial_r^{j+1-k}e^{\lambda_1(r)t}\partial_r^k\chi_{\rm H}(r)r^{s+j}\cos(r|x|)\right|_{r=\frac{1}{|x|}} \\
    &\qquad\quad+ \frac{1}{|x|}\left|\int_{\frac{1}{|x|}}^{\infty}\partial_r\left(\partial_r^{j+1-k}e^{\lambda_1(r)t}\partial_r^k\chi_{\rm H}(r)r^{s+j}\right)\cos(r|x|)dr\right|\\
    &\qquad\lesssim \frac{e^{-ct}}{|x|^{2(\sigma-\sigma_2)+k+s}}.
\end{align*}
    where $c$ is a suitable positive constant. From this we conclude
    \begin{equation*}
        |F_{j,k}| \lesssim \frac{e^{-ct}}{|x|^{2(\sigma-\sigma_2)+k+s}}.
    \end{equation*}
From the above estimates we have
\begin{equation} \label{estimate2.1.1}
    \left\|\frak{F}^{-1}\left(e^{\lambda_1 t}|\xi|^{s}\chi_{\rm H}(|\xi|)\right)\chi_{\rm H}(|x|) \right\|_{L^1} \lesssim e^{-ct}.
\end{equation}
for all $t \geq 1$ and $n = 2m+1, m \geq 0$ with $c$ is a suitable positive constant.\\
\noindent\textit{Let us consider even spatial dimensions $n = 2m, m \geq 1$.}  In this case we shall study
\begin{equation*}
    F(t,x) = c \int_0^{\infty} e^{\lambda_1(r)t} r^{s+2m-1}\chi_{\rm H}(r) \tilde{\mathcal{J}}_{m-1}(r|x|)dr.
\end{equation*}
Indeed, we carry out $m-1$ steps of partial integration to obtain 
 \begin{align*}
     F(t,x) &= \frac{c}{|x|^{2m-2}}\int_0^{\infty} \mathcal{V}^{m-1}\left(e^{\lambda_1(r)t}r^{s+2m-1}\left(\chi_{\rm M}(r)+\chi_{\rm H}(r)\right)\right) \mathcal{J}_0(r|x|)dr \\
     &= \displaystyle \sum_{j=0}^{m-1}\frac{c_j}{|x|^{2m-2}}\int_0^{\infty} \partial_r^{j}\left(e^{\lambda_1(r)t}\chi_{\rm H}(r)\right)r^{s+j+1} \mathcal{J}_0(r|x|) dr\\
     &:= \displaystyle \sum_{j=0}^{m-1} c_j H_{j}(t,x) .
 \end{align*}
 Using the first rule and the fifth rule of the modified Bessel functions we conclude
 \begin{align*}
     H_0(t,x) =& \frac{1}{|x|^{2m-2}}\int_0^{\infty} \left(e^{\lambda_1(r)t}\chi_{\rm H}(r)\right)r^{s+1} \mathcal{J}_0(r|x|) dr\\
     =& \frac{1}{|x|^{2m-2}} \int_0^{\infty} 2e^{\lambda_1(r)t}\chi_{\rm H}(r)r^{s+1} \tilde{\mathcal{J}}_1(r|x|) dr\\
     &- \frac{1}{|x|^{2m-2}} \int_0^{\infty} \partial_r\left(e^{\lambda_1(r)t}\chi_{\rm H}(r)r^{s+2}\right)\tilde{\mathcal{J}}_1(r|x|) dr \\
     =& - \frac{1}{|x|^{2m-2}} \int_0^{\infty} \partial_r\left(e^{\lambda_1(r)t}\chi_{\rm H}(r)r^{s}\right)r^2\tilde{\mathcal{J}}_1(r|x|) dr \\
     =& - \frac{1}{|x|^n} \int_0^{\infty} \partial_r\left(\partial_r\left(e^{\lambda_1(r)t}\chi_{\rm H}(r)r^{s}\right)r\right)\mathcal{J}_0(r|x|) dr.
 \end{align*}
Noticing that
 \begin{equation*}
     \left|\partial_r\left(\partial_r\left(e^{\lambda_1(r)t}\chi_{\rm H}(r)r^{s}\right)r\right)\right| \lesssim e^{-ct} r^{2(\sigma-\sigma_2)+s-1} \text{ for all } t \geq 1
 \end{equation*}
 on the support of $\chi_{\rm M}$ and $\chi_{\rm H}$, we have
 \begin{equation*}
     \left|\int_0^{\frac{1}{|x|}}\partial_r\left(\partial_r\left(e^{\lambda_1(r)t}\chi_{\rm H}(r)r^{s}\right)r\right)\mathcal{J}_0(r|x|)dr \right|  \lesssim \frac{e^{-ct}}{|x|^{2(\sigma-\sigma_2)+s}}.
 \end{equation*}
 Here we have used the estimate $|\mathcal{J}_0(s)| \lesssim 1$ for $s \in [0,1]$. Thanks to the relation $|\mathcal{J}_0(s)| \lesssim s^{-\frac{1}{2}}$, we derive
 \begin{align*}
     \left|\int_{\frac{1}{|x|}}^{\infty}\partial_r\left(\partial_r\left(e^{\lambda_1(r)t}\chi_{\rm H}(r)r^{s}\right)r\right)\mathcal{J}_0(r|x|)dr \right| &\lesssim \frac{e^{-ct}}{|x|^{\frac{1}{2}}}\int_{\frac{1}{|x|}}^1 r^{2(\sigma-\sigma_2)+s-\frac{3}{2}} dr \\
     &\lesssim 
    \begin{cases}
        \f{e^{-ct}}{|x|^{\frac{1}{2}}} &\text{ if } 2(\sigma-\sigma_2)+s > \frac{1}{2},\\
         \f{e^{-ct}\log(|x|)}{|x|^{2(\sigma-\sigma_2)+s}} &\text{ if } 2(\sigma-\sigma_2)+s = \frac{1}{2},\\
         \f{e^{-ct}}{|x|^{2(\sigma-\sigma_2)+s}} &\text{ if } 2(\sigma-\sigma_2)+s < \frac{1}{2}.
    \end{cases}
 \end{align*}
 From the above two estimates one obtains
 \begin{equation*}
     |H_0(t,x)| \lesssim \frac{e^{-ct}}{|x|^{n + \frac{1}{2}\min\{\frac{1}{2},s+2(\sigma-\sigma_2)\}}} \quad \text{ for all } t \geq 1.
 \end{equation*}
 Let $i \in [1, m-1]$ be an integer. Then, the same treatment as we did for $\mathcal{J}_0(t,x)$ implies
 \begin{align*}
     H_j(t,x) &= \frac{1}{|x|^{2m-2}}\int_0^{\infty} \partial_r^{j}\left(e^{\lambda_1(r)t}\chi_{\rm H}(r)\right)r^{s+j+1} \mathcal{J}_0(r|x|) dr\\
     &= -\frac{1}{|x|^n}\int_0^{\infty} \partial_r \left( \partial_r\left(\partial_r^j \left(e^{\lambda_1(r)t}\chi_{\rm H}(r)\right)r^{s+j}\right)r\right) \mathcal{J}_0(r|x|)dr.
 \end{align*}
Observing that
 \begin{equation*}
     \left|\partial_r \left( \partial_r\left(\partial_r^j \left(e^{\lambda_1(r)t}\chi_{\rm H}(r)\right)r^{s+j}\right)r\right)\right| \lesssim e^{-ct} e^{-c_1r^{2(\sigma-\sigma_2)}t} r^{2(\sigma-\sigma_2)+s-1},
 \end{equation*}
 on the support of $\chi_{\rm M}$ and $\chi_{\rm H}$, where $c_1$ and $c$ are suitable positive constant, we can follow the same steps as the estimation for $J_0(t,x)$ to achieve 
 \begin{equation*}
     |H_i(t,x)| \lesssim \frac{e^{-ct}}{|x|^{n+s+2(\sigma-\sigma_2)}},
 \end{equation*}
 for all $i \in [1, m-1]$ and $m \geq 1$. From the $J_i(t,x)$ estimates above with $i \in [0, m-1]$, we conclude
\begin{equation}\label{estimate2.1.2}
    \left\|\frak{F}^{-1}\left(e^{\lambda_1 t}|\xi|^{s}\chi_{\rm H}(|\xi|)\right)\chi_{\rm H}(|x|) \right\|_{L^1} \lesssim e^{-ct},
\end{equation}
for all $t \geq 1$ and $n = 2m, m \geq 1$ with $c$ is a suitable positive constant. From the estimates $(\ref{estimate2.1.1})$ and (\ref{estimate2.1.2}) we arrive at
\begin{equation}\label{estimate2.1.3}
    \left\|\frak{F}^{-1}\left(e^{\lambda_1 t}|\xi|^{s}\chi_{\rm H}(|\xi|)\right)\chi_{\rm H}(|x|) \right\|_{L^1} \lesssim e^{-ct}
\end{equation}
for all $t \geq 1$.
\item[$\bullet$] \textbf{Now we turn to verify (\ref{estimate2.1.100}) for $t \in (0,1]$}. 
Taking account of $|x| \leq \frac{1}{\epsilon^{*}}$, we split the integral
\begin{align*}
    &\int_{\mathbb{R}^n} \frak{F}^{-1}\left(e^{\lambda_1(|\xi|) t}|\xi|^{s}\chi_{\rm H}(|\xi|)\right) (\chi_{\rm L}(|x|) +\chi_{\rm M}(|x|))dx\\
    &\,\,=  \int_{\mathbb{R}^n} \frak{F}^{-1}\left(e^{\lambda_1(|\xi|) t}|\xi|^{s}\chi_{\rm H}(|\xi|)\right) (\chi_{\rm L}(|x|)+\chi_{\rm M}(|x|))\left(\chi_{\rm L}\left(|x|t^{-\frac{1}{2(\sigma-\sigma_2)}}\right)+\chi_{\rm M}\left(|x|t^{-\frac{1}{2(\sigma-\sigma_2)}}\right)\right) dx\\
    &\quad + \int_{\mathbb{R}^n} \frak{F}^{-1}\left(e^{\lambda_1(|\xi|) t}|\xi|^{s}\chi_{\rm H}(|\xi|)\right) (\chi_{\rm L}(|x|)+\chi_{\rm M}(|x|))\chi_{\rm H}\left(|x|t^{-\frac{1}{2(\sigma-\sigma_2)}}\right)dx \\
    &\,\,:= I_1 + I_2 .
\end{align*}
To estimate both integrals, we use the change of variables $\eta= \xi t^{\frac{1}{2(\sigma-\sigma_2)}}$ and $x= y t^{\frac{1}{2(\sigma-\sigma_2)}}$. For the first integral one has
\begin{align}
   |I_1| &\lesssim t^{-\frac{s}{2(\sigma-\sigma_2)}}\int_{\mathbb{R}^n}\chi_{\rm L}(|y|) +\chi_{\rm M}(|y|)dy \int_{\mathbb{R}^n}e^{-|\eta|^{2(\sigma-\sigma_2)}}|\eta|^{s} d\eta \notag\\
    &\lesssim t^{-\frac{s}{2(\sigma-\sigma_2)}}. \label{estimate2.1.4}
\end{align}
For the second integral, using the modified Bessel functions we get 
\begin{align*}
|I_2| &\lesssim t^{-\frac{s}{2(\sigma-\sigma_2)}}\int_0^{\infty} \left(\chi_{\rm L}\left(|y|t^{\frac{1}{2(\sigma-\sigma_2)}} \right)+\chi_{\rm M}\left(|y|t^{\frac{1}{2(\sigma-\sigma_2)}} \right)\right)\chi_{\rm H}(|y|) \\ 
&\qquad\qquad\qquad \times \left|\int_0^{\infty}e^{\lambda_1(r t^{-\frac{1}{2(\sigma-\sigma_2)}})t}r^{s+n-1}\chi_{\rm H}\left(rt^{-\frac{1}{2(\sigma-\sigma_2)}}\right) \cos(r|y|)dr\right| dy.
\end{align*}
\textit{Let us consider odd spatial dimensions $n = 2m+1, m \geq 1$.} Using the vector field $\mathcal{V}(f(r))$ we carry out $m+1$ steps of partial integration to obtain
\begin{align}
    P(t,y) :=& p_1\int_0^{\infty}e^{\lambda_1(r t^{-\frac{1}{2(\sigma-\sigma_2)}})t}r^{s+n-1}\chi_{\rm H}\left(rt^{-\frac{1}{2(\sigma-\sigma_2)}}\right) \tilde{\mathcal{J}}_{\frac{n}{2}-1}(r|y|) dr \label{equation2.1.1}\\
    =& \frac{p_1}{|y|^n} \int_0^{\infty} \partial_r\left(\mathcal{V}^m \left(e^{\lambda_1(r t^{-\frac{1}{2(\sigma-\sigma_2)}})t}r^{s+n-1}\chi_{\rm H}\left(rt^{-\frac{1}{2(\sigma-\sigma_2)}}\right)\right)\right) \sin(r|y|)dr \notag\\
    =& \displaystyle \sum_{j=0}^m \displaystyle \sum_{k=0}^{j+1} \frac{p_{jk}}{|y|^n} \int_0^{\infty} \partial_r^{j+1-k} e^{\lambda_1(r t^{-\frac{1}{2(\sigma-\sigma_2)}})t}\partial_r^k\chi_{\rm H}\left(rt^{-\frac{1}{2(\sigma-\sigma_2)}}\right)r^{s+j}\sin(r|y|)dr \notag\\
    &+  \displaystyle \sum_{j=0}^m \displaystyle \sum_{k=0}^{j} \frac{p_{jk}}{|y|^n} \int_0^{\infty} \partial_r^{j-k} e^{\lambda_1(r t^{-\frac{1}{2(\sigma-\sigma_2)}})t}\partial_r^{k+1}\chi_{\rm H}\left(rt^{-\frac{1}{2(\sigma-\sigma_2)}}\right) r^{s+j}\sin(r|y|)dr \notag\\
    &+  \displaystyle \sum_{j=1}^m \displaystyle \sum_{k=0}^{j} \frac{p_{jk}}{|y|^n} \int_0^{\infty} \partial_r^{j-k} e^{\lambda_1(r t^{-\frac{1}{2(\sigma-\sigma_2)}})t}\partial_r^{k}\chi_{\rm H}\left(rt^{-\frac{1}{2(\sigma-\sigma_2)}}\right) r^{s+j-1}\sin(r|y|)dr \notag.
\end{align} 
For this reason we will consider integrals
\begin{align*}
    |P_{j,k}(r,y)| &= \frac{1}{|y|^n} \left|\int_0^{\infty} \partial_r^{j-k} e^{\lambda_1(r t^{-\frac{1}{2(\sigma-\sigma_2)}})t}\partial_r^{k}\chi_{\rm H}\left(rt^{-\frac{1}{2(\sigma-\sigma_2)}}\right) r^{s+j-1}\sin(r|y|)dr\right|,
\end{align*}
where $j \geq k$ and $(j,k) \ne (0,0)$. First, we consider the case $j > k \geq 0$. Applying Corollary \ref{corollary1}, splitting the above integral into two parts we get
\begin{align*}
    &\left|\int_0^{\frac{1}{|y|}} \partial_r^{j-k} e^{\lambda_1(r t^{-\frac{1}{2(\sigma-\sigma_2)}})t}\partial_r^{k}\chi_{\rm H}\left(rt^{-\frac{1}{2(\sigma-\sigma_2)}}\right) r^{s+j-1}\sin(r|y|)dr\right| \\
    &\qquad\lesssim \int_0^{\frac{1}{|y|}} e^{-c_1 r^{2(\sigma-\sigma_2)}} r^{s+2(\sigma-\sigma_2)-1} dr\\
    &\qquad\lesssim \frac{1}{|y|^{s+2(\sigma-\sigma_2)}}.
\end{align*}
On the other hand, performing integration by parts we obtain
\begin{align*}
    &\left|\int_{\frac{1}{|y|}}^{\infty} \partial_r^{j-k} e^{\lambda_1(r t^{-\frac{1}{2(\sigma-\sigma_2)}})t}\partial_r^{k}\chi_{\rm H}\left(rt^{-\frac{1}{2(\sigma-\sigma_2)}}\right) r^{s+j-1}\sin(r|y|)dr\right|\\
    &\qquad\lesssim \frac{1}{|y|^{s+2(\sigma-\sigma_2)}} + \frac{1}{|y|} \left|\int_{\frac{1}{|y|}}^{\infty} \partial_r\left(\partial_r^{j-k} e^{\lambda_1(r t^{-\frac{1}{2(\sigma-\sigma_2)}})t}\partial_r^{k}\chi_{\rm H}\left(rt^{-\frac{1}{2(\sigma-\sigma_2)}}\right) r^{s+j-1}\right)\cos(r|y|)dr\right|\\
    &\qquad\lesssim \frac{1}{|y|^{s+2(\sigma-\sigma_2)}} + \frac{1}{|y|}\int_{\frac{1}{|y|}}^{\infty} e^{-c_1r{2(\sigma-\sigma_2)}} r^{s+2(\sigma-\sigma_2)-2} dr\\
    &\qquad\lesssim \frac{1}{|y|^{s+2(\sigma-\sigma_2
    )}} + \frac{1}{|y|}\int_{\frac{1}{|y|}}^1 e^{-c_1r{2(\sigma-\sigma_2)}} r^{s+2(\sigma-\sigma_2)-2} dr\\
    &\qquad\lesssim
    \begin{cases}
        \displaystyle\frac{1}{|y|} &\text{ if } s+2(\sigma-\sigma_2) > 1,\\
       \displaystyle \frac{\log(|y|)}{|y|} &\text{ if } s+2(\sigma-\sigma_2) = 1,\\
        \displaystyle\frac{1}{|y|^{s+2(\sigma-\sigma_2)}} &\text{ if } s +2(\sigma-\sigma_2) < 1.
    \end{cases}
\end{align*}
Here we note that $r \sim t^{\frac{1}{2(\sigma-\sigma_2)}}$ on the support of the derivative of $\chi_{\rm H}\left(rt^{-\frac{1}{2(\sigma-\sigma_2)}}\right)$. From here we obtain
\begin{align*}
    |P_{j,k}(t,y)| \lesssim
    \begin{cases}
        \displaystyle\frac{1}{|y|^{n+1}} &\text{ if } s+2(\sigma-\sigma_2) > 1,\\
       \displaystyle \frac{\log(|y|)}{|y|^{n+1}} &\text{ if } s+2(\sigma-\sigma_2) = 1,\\
        \displaystyle\frac{1}{|y|^{n+s+2(\sigma-\sigma_2)}} &\text{ if } s +2(\sigma-\sigma_2) < 1,
    \end{cases}
\end{align*}
for all $j > k > 0$. Now we will consider the case $j = k \geq 1$ with integral
\begin{align*}
    |P_{k,k}(r,y)| =&\frac{1}{|y|^n} \left|\int_0^{\infty} e^{\lambda_1(rt^{-\frac{1}{2(\sigma-\sigma_2)}})t}\partial_r^k\chi_{\rm H}\left(rt^{-\frac{1}{2(\sigma-\sigma_2)}}\right)r^{s+k-1}\sin(r|y|)dr\right|.
\end{align*}
From relations $|y|t^{\frac{1}{2(\sigma-\sigma_2)}} \leq \frac{2}{\epsilon^{*}}$ we have 
\begin{align*}
    |P_{k,k}(r,y)| &\lesssim \left|\int_0^{\frac{1}{|y|}} e^{\lambda_1(rt^{-\frac{1}{2(\sigma-\sigma_2)}})t}\partial_r^k\chi_{\rm H}\left(rt^{-\frac{1}{2(\sigma-\sigma_2)}}\right)r^{s+k-1}\sin(r|y|)dr \right|\\
    &\lesssim
    \begin{cases}
        \displaystyle\frac{1}{|y|^s} &\text{ if } s > 0,\\
        \displaystyle -t^{\frac{a}{2(\sigma-\sigma_2)}}|y|^a\log\left(|y|t^{\frac{1}{2(\sigma-\sigma_2)}}\right) &\text{ if } s = 0.
    \end{cases}
\end{align*}
In the case $s = 0$, we note that $|\sin(\nu)| \leq \nu^a$ for all $\nu \geq 0$ with $0 < a < 1$. From here we obtain
\begin{align*}
    |I_2| 
    &\lesssim
    \begin{cases}
        t^{-\frac{s}{2(\sigma-\sigma_2)}} \displaystyle\int_0^{\infty} \frac{\chi_{\rm H}(|y|)}{|y|^{1+s}} d|y| &\text{ if } s > 0\\
        -\displaystyle \int_0^{\infty} t^{\frac{a}{2(\sigma-\sigma_2)}}\left(\chi_{\rm L}\left(|y|t^{\frac{1}{2(\sigma-\sigma_2)}} \right)+\chi_{\rm M}\left(|y|t^{\frac{1}{2(\sigma-\sigma_2)}}  \right)\right) \frac{\log\left(|y|t^{\frac{1}{2(\sigma-\sigma_2)}}\right)}{|y|^{1-a}} d|y| &\text{ if } s = 0
    \end{cases} \\
    &\lesssim t^{-\frac{s}{2(\sigma-\sigma_2)}}\text{ for all }s\geq 0.
\end{align*}

\textit{Let us consider even spatial dimensions $n = 2m, m \geq 1$.} Carrying out $m -1$ steps of partial integration we re-write (\ref{equation2.1.1}) as follows:
\begin{align}
    J(t,y) &= c_1 \int_0^{\infty} e^{\lambda_1(r t^{-\frac{1}{2(\sigma-\sigma_2)}})t} r^{s+n-1} \chi_{\rm H}\left(rt^{-\frac{1}{2(\sigma-\sigma_2)}}\right) \tilde{\mathcal{J}}_{m-1}(r|y|)dr \notag\\
    &= \frac{c_1}{|y|^{2m-2}} \int_0^{\infty} \mathcal{V}^{m-1} \left(e^{\lambda_1(r t^{-\frac{1}{2(\sigma-\sigma_2)}})t}r^{s+2m-1}\chi_{\rm H}\left(rt^{-\frac{1}{2(\sigma-\sigma_2)}}\right)\right)\mathcal{J}_0(r|y|)dr \notag\\
    &= \displaystyle \sum_{j=0}^{m-1} \frac{a_j}{|y|^{2m-2}}\int_0^{\infty}\partial_r^j\left(e^{\lambda_1(r t^{-\frac{1}{2(\sigma-\sigma_2)}})t}\chi_{\rm H}\left(rt^{-\frac{1}{2(\sigma-\sigma_2)}}\right)\right) r^{s+j+1} \mathcal{J}_0(r|y|)dr \notag\\
    &= \displaystyle \sum_{j=0}^{m-1} \frac{a_j}{|y|^{n}}\int_0^{\infty} \partial_r\left(\partial_r\left(\partial_r^j\left(e^{\lambda_1(rt^{-\frac{1}{2(\sigma-\sigma_2)}})t}\chi_{\rm H}\left(rt^{-\frac{1}{2(\sigma-\sigma_2)}}\right)\right)r^{s+j}\right)r\right) \mathcal{J}_0(r|y|) dr \label{equation2.1.2}\\
     &= \displaystyle \sum_{j=0}^{m-1} \frac{a_j}{|y|^{n}} L_j.\notag
 \end{align}
   The equality (\ref{equation2.1.2}) occurs because we use the first rule of moified Bessel functions for $\mu = 1$ and the fifth rule for $\mu = 0$. Splitting the integral $L_j$ with into two parts, we get the first part
\begin{align*}
     & \left|\int_0^{\frac{1}{|y|}} \partial_r\left(\partial_r\left(\partial_r^j\left(e^{\lambda_1(rt^{-\frac{1}{2(\sigma-\sigma_2)}})t}\chi_{\rm H}\left(rt^{-\frac{1}{2(\sigma-\sigma_2)}}\right)\right)r^{s+j}\right)r\right) \mathcal{J}_0(r|y|) dr\right|\\
     &\qquad\lesssim \left|\int_0^{\frac{1}{|y|}} \partial_r\left(\partial_r\left(\partial_r^j\left(e^{\lambda_1(rt^{-\frac{1}{2(\sigma-\sigma_2)}})t}\chi_{\rm H}\left(rt^{-\frac{1}{2(\sigma-\sigma_2)}}\right)\right)r^{s+j}\right)r\right) (1-\mathcal{J}_0(r|y|)) dr\right|\\
     &\qquad\quad +
     \begin{cases}
         \displaystyle\frac{1}{|y|^s} &\text{ if } s > 0, \\
         \displaystyle\left|\chi_{\rm H}^{(j+1)}\left(|y|^{-1}t^{-\frac{1}{2(\sigma-\sigma_2)}}\right)\right| &\text{ if } s=0, j \geq 0.
     \end{cases}
\end{align*}
Integrals containing the derivative of $e^{\lambda_1(rt^{-\frac{1}{2(\sigma-\sigma_2)}})t}$ are evaluated as the case $n$ were odd. Using the relation $1-J_0(\nu) \leq \nu^a$ with $\nu \in [0,1]$ and a is suitable positive constant in $(0,1)$ we can evaluate the remaining integrals as follows
\begin{align*}
    &\left|\int_0^{\frac{1}{|y|}} e^{\lambda_1(rt^{-\frac{1}{2(\sigma-\sigma_2)}})t} \partial_r^{j+2} \chi_{\rm H}\left(rt^{\frac{-1}{2(\sigma-\sigma_2)}}\right) r^{s+j+1}(1-\mathcal{J}_0(r|y|))dr\right|\\
    &\qquad\lesssim  \int_0^{\frac{1}{|y|}} r^{s-1} \chi_{\rm H}^{(j+2)}\left(rt^{-\frac{1}{2(\sigma-\sigma_2)}}\right)|1-\mathcal{J}_0(r|y|)|dr\\
    &\qquad\lesssim
    \begin{cases}
        \displaystyle\frac{1}{|y|^s} &\text{ if } s > 0,\\
        -|y|^a t^{\frac{a}{2(\sigma-\sigma_2)}} \log\left(|y|t^{\frac{1}{2(\sigma-\sigma_2)}}\right) &\text{ if } s = 0.
    \end{cases}
\end{align*}
On the other hand
\begin{align*}
    &\left|\int_{\frac{1}{|y|}}^1 e^{\lambda_1(rt^{-\frac{1}{2(\sigma-\sigma_2)}})t} \partial_r^{j+2} \chi_{\rm H}\left(rt^{\frac{-1}{2(\sigma-\sigma_2)}}\right) r^{s+j+1}\mathcal{J}_0(r|y|)dr\right|\\
    &\qquad\lesssim \int_{\frac{1}{|y|}}^1 r^{s-1} \chi_{\rm H}^{(j+2)}\left(rt^{-\frac{1}{2(\sigma-\sigma_2)}}\right) |\mathcal{J}_0(r|y|)| dr\\
    &\qquad\lesssim 
    \begin{cases}
        \displaystyle\frac{1}{|y|^s} &\text{ if } s > 0,\\
        -|y|^a t^{\frac{a}{2(\sigma-\sigma_2)}} \log\left(|y|t^{\frac{1}{2(\sigma-\sigma_2)}}\right) &\text{ if } s = 0.
    \end{cases}
\end{align*}

Here we have used the estimate $|\mathcal{J}_0(h)| \lesssim 1$ for $s \in [0,1]$ and $|\mathcal{J}_0(h)| \lesssim h^{-\frac{1}{2}} \leq h^a$ for $h \geq 1$ and $a \in (0,1)$. From the above two estimates and and perform the same proof steps as if n is an odd number, we have
  \begin{equation*}
      |I_2| \lesssim t^{-\frac{s}{2(\sigma-\sigma_2)}}
  \end{equation*}
  where every n is an even natural number. Linking this estimate to (\ref{equation2.1.2}) we obtain
\begin{equation*}\label{estimate2.1.6}
    \left\|\frak{F}^{-1}\left(e^{\lambda_1 t}|\xi|^{s}\chi_{\rm H}(|\xi|)\right)\chi_{\rm L}(|x|) \right\|_{L^1} \lesssim t^{-\frac{s}{2(\sigma-\sigma_2)}},
\end{equation*}
for all $t \in (0, 1]$. Finally, we prove the estimate (\ref{estimate2.1.100}) for $|x| > \frac{1}{\epsilon^{*}}$. We have the following relations:
\begin{align*}
    &\left\|\frak{F}^{-1}\left(e^{\lambda_1(|\xi|)t}|\xi|^s \chi_{\rm H}(|\xi|)\right)\chi_{\rm H}(|x|)\right\|_{L^1}\\
    &\qquad\lesssim \left\|\frak{F}^{-1}\left(e^{\lambda_1(|\xi|)t}|\xi|^s \chi_{\rm H}(|\xi|)\right)\chi_{\rm H}(|x|)\left(\chi_{\rm L}\left(|x|t^{-\frac{1}{2(\sigma-\sigma_2)}}\right)+\chi_{\rm M}\left(|x|t^{-\frac{1}{2(\sigma-\sigma_2)}}\right)\right)\right\|_{L^1}\\
     &\qquad\quad+ \left\|\frak{F}^{-1}\left(e^{\lambda_1(|\xi|)t}|\xi|^s \chi_{\rm H}(|\xi|)\right)\chi_{\rm H}(|x|)\chi_{\rm H}\left(|x|t^{-\frac{1}{2(\sigma-\sigma_2)}}\right)\right\|_{L^1}\\
    &\qquad:= Q_1 + Q_2.
\end{align*}
Using the change of variables $\eta= \xi t^{\frac{1}{2(\sigma-\sigma_2)}}$ and $x= y t^{\frac{1}{2(\sigma-\sigma_2)}}$ we obtain
\begin{align*}
    Q_1 &\lesssim t^{-\frac{s}{2(\sigma-\sigma_2)}} \int_0^{\infty} (\chi_{\rm L}(|y|)+\chi_{\rm M}(|y|)dy \int_{\mathbb{R}^n} e^{-r^{2(\sigma-\sigma_2)}} r^s dr\\
    &\lesssim t^{-\frac{s}{2(\sigma-\sigma_2)}}.
\end{align*}
Now we estimate $Q_2$. Using the modified Bessel functions we get 
\begin{align*}
Q_2 &\lesssim t^{-\frac{s}{2(\sigma-\sigma_2)}}\int_0^{\infty} \chi_{\rm H}\left(|y|t^{\frac{1}{2(\sigma-\sigma_2)}} \right)\chi_{\rm H}(|y|) \\ 
&\qquad\qquad\qquad \times \left|\int_0^{\infty}e^{\lambda_1(r t^{-\frac{1}{2(\sigma-\sigma_2)}})t}r^{s+n-1}\chi_{\rm H}\left(rt^{-\frac{1}{2(\sigma-\sigma_2)}}\right) \cos(r|y|)dr\right| dy.
\end{align*}
\textit{Let us consider odd spatial dimensions $n = 2m+1, m \geq 1$.} Using the vector field $\mathcal{V}(f(r))$ we carry out $m+1$ steps of partial integration to obtain relations (\ref{equation2.1.1}). Carrying out the same proof steps as in the case $|x| \leq \frac{1}{\epsilon^{*}}$ for integrals containing derivatives of $e^{\lambda_1(rt^{-\frac{1}{2(\sigma-\sigma_2)}})t}$ we get the necessary estimates. For this reason we only need to consider integrals
\begin{align*}
    |P_{k,k}(r,y)| =&\frac{1}{|y|^n} \left|\int_0^{\infty} e^{\lambda_1(rt^{-\frac{1}{2(\sigma-\sigma_2)}})t}\partial_r^k\chi_{\rm H}\left(rt^{-\frac{1}{2(\sigma-\sigma_2)}}\right)r^{s+k-1}\sin(r|y|)dr\right|,
\end{align*}
where $k \geq 1$. Perform splitting the integral into two parts. For the first part, using relation $|\sin(\nu)| \leq \nu^{-a}$ for all $a, \nu \in (0,1)$ we obtain
\begin{align*}
   &\left| \int_0^{\frac{1}{|y|}} e^{\lambda_1(rt^{-\frac{1}{2(\sigma-\sigma_2)}})t}\partial_r^k\chi_{\rm H}\left(rt^{-\frac{1}{2(\sigma-\sigma_2)}}\right)r^{s+k-1}\sin(r|y|)dr\right|\\
   &\qquad \lesssim
   \begin{cases}
       \displaystyle\frac{1}{|y|^s} &\text{ if } s > 0,\\
       t^{-\frac{a}{2(\sigma-\sigma_2)}}|y|^{-a} \log\left(|y|t^{\frac{1}{2(\sigma-\sigma_2)}}\right) &\text{ if } s = 0.
   \end{cases}
\end{align*}
In the second part, performing integration by parts we get
\begin{align*}
    &\left|\int_{\frac{1}{|y|}}^1 e^{\lambda_1(rt^{-\frac{1}{2(\sigma-\sigma_2)}})t}\partial_r^k\chi_{\rm H}\left(rt^{-\frac{1}{2(\sigma-\sigma_2)}}\right)r^{s+k-1}\sin(r|y|)dr\right|\\
    &\qquad\lesssim \frac{1}{|y|} \left|\int_\frac{1}{|y|}^1 \partial_r\left(e^{\lambda_1(rt^{-\frac{1}{2(\sigma-\sigma_2)}})t}\partial_r^k\chi_{\rm H}\left(rt^{-\frac{1}{2(\sigma-\sigma_2)}}\right)r^{s+k-1}\right)\cos(r|y|)dr \right|\\
    &\qquad\quad+
    \begin{cases}
        \displaystyle\frac{1}{|y|^s} &\text{ if } s > 0,\\
        \displaystyle\frac{1}{|y|} + \left(|y|t^{\frac{1}{2(\sigma-\sigma_2)}}\right)^{-k
        } \chi_{\rm H}^{(k)}\left(|y|^{-1}t^{-\frac{1}{2(\sigma-\sigma_2)}}\right) &\text{ if } s =0.
    \end{cases}
\end{align*}
Arguing as above, we only need to evaluate the integral that does not contain the derivative of $e^{\lambda_1\left(rt^{-\frac{1}{2(\sigma-\sigma_2)}}t\right)}$, this integral is evaluated as follows:
\begin{align*}
    &\frac{1}{|y|}\left|\int_{\frac{1}{|y|}}^1 e^{\lambda_1(rt^{-\frac{1}{2(\sigma-\sigma_2)}})t}\partial_r^{k+1}\chi_{\rm H}\left(rt^{-\frac{1}{2(\sigma-\sigma_2)}}\right)r^{s+k-1}\cos(r|y|)dr\right|\\
    &\qquad\lesssim
    \begin{cases}
        \displaystyle\frac{1}{|y|^s} &\text{ if } s > 0,\\
        \displaystyle |y|^{-1}t^{-\frac{1}{2(\sigma-\sigma_2)}} &\text{ if } s=0.
    \end{cases}
\end{align*}
From here we have the following estimated conclusion
\begin{align*}
    |P_{k,k}(r,y)| \lesssim
    \begin{cases}
        \displaystyle\frac{1}{|y|^{n+s}} &\text{ if } s > 0,\\
        t^{-\frac{1}{2(\sigma-\sigma_2)}}|y|^{-n-1} + t^{-\frac{a}{2(\sigma-\sigma_2)}}|y|^{-n-a} \log\left(|y|t^{\frac{1}{2(\sigma-\sigma_2)}}\right) \\
        \qquad\qquad\qquad\qquad\qquad\qquad\qquad+ |y|^{-n}\chi_{\rm H}^{(k)}\left(|y|^{-1}t^{-\frac{1}{2(\sigma-\sigma_2)}}\right) &\text{ if } s = 0,
    \end{cases}
\end{align*}
for all $k \geq 1$. Hence we obtain
\begin{align*}
    Q_2 \lesssim t^{-\frac{s}{2(\sigma-\sigma_2)}} \text{ for all } s \geq 0.
\end{align*}
\textit{Let us consider even spatial dimensions $n = 2m, m \geq 1$.} Perform the same proof steps as above, noting that $1-\mathcal{J}_0(\nu) \leq \nu^{-a}$ where $\nu \in (0,1)$ and $a$ is a suitable positive constant in $(0,1)$, we also get the estimate we need. The proof of Lemma \ref{Lemma2.1} is complete.
\end{itemize}
\end{proof}

Following the proof of Lemma 
\ref{Lemma2.1} we may conclude the following $L^1$ estimates, too.
\begin{lemma}\label{Lemma2.2}
    The following estimates hold in $\mathbb{R}^n$ for any $s \geq 0$:
    \begin{align}
        \left\|\frak{F}^{-1}\left(|\xi|^{s}e^{\lambda_1(|\xi|) t}\chi_{\rm L}(|\xi|)\right)(t,\cdot)\right\|_{L^1} \lesssim 
    (1+t)^{-\frac{s}{2(\sigma-\sigma_1)}}\label{estimate2.2.100}
    \end{align}
    for all $n \geq 1$ and $t > 0$.
\end{lemma}
Before proving Lemma \ref{Lemma2.2}, we need to prove the following auxiliary proposition:
\begin{proposition}\label{proposition2.2}
    The following estimate hold:
    \begin{align*}
        \left|\partial_{\eta}^je^{\lambda_1(\eta)}\right|\chi_{\rm L}(|\eta|) \lesssim \eta^{2(\sigma-\sigma_1)-j}
    \end{align*}
    for all $j \geq 1$.
\end{proposition}
\begin{proof}
    Perform the same proof steps in Proposition \ref{proposition2.1} with note that
    \begin{align*}
        \lambda_1(\eta) = -\eta^{2(\sigma-\sigma_1)}\left(\frac{1}{1+\eta^{2(\sigma_2-\sigma_1)}}\frac{1}{\frac{1}{2}+\frac{1}{2}\sqrt{1-\frac{4\eta^{2(\sigma-2\sigma_1)}}{\left(1+\eta^{2(\sigma_2-\sigma_1)}\right)^2}}}\right),
    \end{align*}
    we have the conclusion of Proposition \ref{proposition2.2}.
\end{proof}
Similar to Corollary \ref{corollary1}, we also get the following estimate.
\begin{corollary}
    The following estimate hold:
    \begin{align*}
        \left|\partial_r^je^{\lambda_1(rt^{-\frac{1}{2(\sigma-\sigma_1)}})t}\right| \chi_{\rm L}\left(rt^{-\frac{1}{2(\sigma-\sigma_1)}}\right) \lesssim e^{-cr^{2(\sigma-\sigma_1)}}r^{2(\sigma-\sigma_1)-j},
    \end{align*}
    for all $j \geq 1$ and $c$ is suitable positive constant.
\end{corollary}
\begin{proof}[Proof of Lemma \ref{Lemma2.2}]
    In the case $|x| \leq \frac{1}{\epsilon^{*}}$, then we have immediately
    \begin{align*}
        \left|\frak{F}^{-1}\left(|\xi|^{s}e^{\lambda_1(|\xi|)t}(\chi_{\rm L}(|\xi|)+\chi_{\rm M}(|x|))\right)\chi_{\rm L}(|x|)\right|_{L^1} &\lesssim \int_{\mathbb{R}^n}\chi_{\rm L}(|x|)dx \int_{\mathbb{R}^n} |\xi|^{s} e^{-|\xi|^{2(\sigma-\sigma_1)}t} \chi_{\rm L}(|\xi|) d\xi \\
        &\lesssim 
        (1+t)^{-\frac{s}{2(\sigma-\sigma_1)}}
    \end{align*}
    for all $t > 0$.
    For this reason, it suffices to prove Lemma \ref{Lemma2.2} in the remaining case $|x| > \frac{1}{\epsilon^{*}}$. Let us separate our consideration into two subcases as follows:
    \begin{itemize}[leftmargin=*]
    \item[$\bullet$] \textbf{At first we consider the case $t \in (0,1]$}. Let us introduce the function
    \begin{equation*}
        G(t,x) := \frak{F}^{-1}\left(e^{\lambda_1(|\xi|)t}|\xi|^{s}\chi_{\rm L}(|\xi|)\right)(t,x).
    \end{equation*}
    Using the radial symmetry of the integrand and modified Bessel functions we have to consider
    \begin{equation}\label{equation2.2.1}
        G(t,x):= a^{*} \int_0^{\infty} e^{\lambda_1(r)t} r^{s+n-1} \chi_{\rm L}(r) \tilde{\mathcal{J}}_{\frac{n}{2}-1}(r|x|)dr.
    \end{equation}
   \textit{Let us consider odd spatial dimensions $n= 2m+1, m \geq 1$.} Again, using the vector field $\mathcal{V}(f(r))$ we carry out $m+1$ steps of partial integration to obtain
   \begin{align*}
       G(t,x) =& -\frac{a^{*}}{|x|^n} \int_0^{\infty} \partial_r\left(\mathcal{V}^m\left(e^{\lambda_1(r)t}r^{s+2m}\chi_{\rm L}(r)\right)\right) \sin(r|x|)dr\\
       =& \displaystyle \sum_{j=0}^m \displaystyle \sum_{k=0}^{j+1} \frac{a_{jk}}{|x|^n} \int_0^{\infty} \partial_r^{j+1-k} e^{\lambda_1(r)t} \partial_r^k(\chi_{\rm L}(r))r^{s+j} \sin(r|x|)dr\\
       &+ \displaystyle \sum_{j=0}^m \displaystyle \sum_{k=0}^j \frac{a_{jk}}{|x|^n} \int_0^{\infty} \partial_r^{j-k} e^{\lambda_1(r)t} \partial_r^{k+1}(\chi_{\rm L}(r)) r^{s+j} \sin(r|x|)dr \\
       &+ \displaystyle \sum_{j=0}^m \displaystyle \sum_{k=0}^j \frac{a_{jk}}{|x|^n} \int_0^{\infty} \partial_r^{j-k} e^{\lambda_1(r)t}\partial_r^k(\chi_{\rm L}(r))r^{s+j-1} \sin(r|x|)dr
   \end{align*}
  with some constants $a_{jk}$. For this reason, we only need to study the integrals
  \begin{align*}
      G_{j,k} &:= \frac{1}{|x|^n} \int_0^{\infty} \partial_r^{j+1-k} e^{\lambda_1(r)t} \partial_r^k(\chi_{\rm L}(r))r^{s+j}\sin(r|x|)dr.
  \end{align*}
Namely, one gets
 \begin{align*}
     |G_{j,k}| \lesssim& \frac{1}{|x|^{n+1}}\left|\partial_r^{j+1-k} e^{\lambda_1(r)t} \partial_r^k(\chi_{\rm L}(r))r^{s+j}\cos(r|x|)\right|_{r=0}^{\infty}\\
     &+ \frac{1}{|x|^{n+1}}\int_0^{\infty} \partial_r\left(\partial_r^{j+1-k} e^{\lambda_1(r)t} \partial_r^k(\chi_{\rm L}(r))r^{s+j}\right) \cos(r|x|) dr\\
     \lesssim& \frac{1}{|x|^{n+1}},
 \end{align*}
 for all $t > 0$ and $k \geq 1$. Now we consider the case $k = 0$. By splitting the integral $G_{j,0}$ into two parts, on the one hand we obtain the following estimate:
 \begin{align*}
     \frac{1}{|x|^n} \left|\int_0^{\frac{1}{|x|}} \partial_r^{j+1} e^{\lambda_1(r)t}\chi_{\rm L}(r)r^{s+j}\sin(r|x|)dr\right| &\lesssim \frac{1}{|x|^n} \int_0^{\frac{1}{|x|}} r^{2(\sigma-\sigma_1)+s-1} dr\\
     &\lesssim \frac{1}{|x|^{n+2(\sigma-\sigma_1)+s}}.
 \end{align*}
 On the other hand, carrying out one more step of partial integration we derive
 \begin{align*}
     \frac{1}{|x|^n}\left|\int_{\frac{1}{|x|}}^{\infty} \partial_r^{j+1} e^{\lambda_1(r)t}\chi_{\rm L}(r)r^{s+j} \sin(r|x|)dr\right| \lesssim& \frac{1}{|x|^{n+1}}\left|\partial_r^{j+1} e^{\lambda_1(r)t}\chi_{\rm L}(r)r^{s+j} \cos(r|x|)\right|_{r=\frac{1}{|x|}}\\
     &+ \frac{1}{|x|^{n+1}} \left|\int_{\frac{1}{|x|}}^{\infty}\partial_r\left(\partial_r^{j+1} e^{\lambda_1(r)t}\chi_{\rm L}(r)r^{s+j} \right)\cos(r|x|)dr\right| \\
     \lesssim& \frac{1}{|x|^{n+1}}. 
 \end{align*}
Thus, it follows that
 \begin{equation}\label{equation2.2.2}
     \left\|\frak{F}^{-1}\left(e^{\lambda_1(|\xi|)t}|\xi|^{s}\chi_{\rm L}(|\xi|)\right)(\chi_{\rm M}(|x|)+\chi_{\rm H}(|x|))\right\|_{L^1} \lesssim 1,
 \end{equation}
 for all $t \in (0,1]$ and $n$ is an odd natural number.\\
 \textit{Let us consider even spatial dimensions $n= 2m, m \geq 1$.} Carrying out $m-1$ steps of partial integration we re-write (\ref{equation2.2.1}) as follows:
 \begin{align*}
     G(t,x) =& \frac{a^{*
     }}{|x|^{2m-2}} \int_0^{\infty} \mathcal{V}^{m-1} \left(e^{\lambda_1(r)t}r^{s+2m-1}\chi_{\rm L}(r)\right) J_0(r|x|)dr \\
     =& \displaystyle \sum_{j=0}^{m-1} \frac{a_j}{|x|^{2m-2}} \int_0^{\infty} \partial_r^j\left(e^{\lambda_1(r)t}r^{s}\chi_{\rm L}(r)\right)r^{j+1}J_0(r|x|) dr \\
     :=& \displaystyle \sum_{j=0}^{m-1}a_jG_j(t,x).
 \end{align*}
 Using the first rule of modified Bessel functions for $\mu = 1$ and the fifth rule for $\mu = 0$ and the performing two more steps of partial integration we get
 \begin{align*}
     G_j(t,x) = -\frac{1}{|x|^{n}} \int_0^{\infty} \partial_r\left(\partial_r\left(\partial_r^j\left(e^{\lambda_1(r)t}r^{s}\chi_{\rm L}(r)\right)r^j\right)r\right)J_0(r|x|)dr.
 \end{align*}
 By splitting the integral $G_j$ into two
parts, on the one hand we obtain the following estiamate
 \begin{align*}
    \frac{1}{|x|^{n}}\left| \int_0^{\frac{1}{|x|}}  \partial_r\left(\partial_r\left(\partial_r^j\left(e^{\lambda_1(r)t}r^{s}\chi_{\rm L}(r)\right)r^j\right)r\right)J_0(r|x|) dr\right| &\lesssim \frac{1}{|x|^{n}} \int_0^{\frac{1}{|x|}} r^{2(\sigma-\sigma_1)+s+j-1} dr\\
    &\lesssim \frac{1}{|x|^{n+2(\sigma-\sigma_1)+s+j}}.
 \end{align*}
 On the other hand, carrying out one more step of partial integration we derive
 \begin{align*}
     &\frac{1}{|x|^{n}} \left|\int_{\frac{1}{|x|}}^{\infty} \partial_r\left(\partial_r\left(\partial_r^j\left(e^{\lambda_1(r)t}r^{s}\chi_{\rm L}(r)\right)r^j\right)r\right)J_0(r|x|) dr\right| \\
     &\qquad \lesssim \frac{1}{|x|^{n+\frac{1}{2}}} \int_{\frac{1}{|x|}}^{2a} r^{2(\sigma-\sigma_1)+s+j-\frac{3}{2}} dr \\
     &\qquad \lesssim
     \begin{cases}
         \f{1}{|x|^{n+\frac{1}{2}}} &\text{ if } 2(\sigma-\sigma_1)+s+j > \frac{1}{2},\\
         \f{\log(|x|)}{|x|^{n+\frac{1}{2}}} &\text{ if } 2(\sigma-\sigma_1) +s+j = \frac{1}{2},\\
         \f{1}{|x|^{n+2(\sigma-\sigma_1)+s+j}} &\text{ if } 2(\sigma-\sigma_1) +s+j < \frac{1}{2}.
     \end{cases}
 \end{align*}
From here, we have the following conclusion:
 \begin{equation}\label{equation2.2.3}
      \left\|\frak{F}^{-1}\left(e^{\lambda_1(|\xi|)t}|\xi|^{s}\chi_{\rm L}(|\xi|)\right)(\chi_{\rm M}(|x|)+\chi_{\rm H}(|x|))\right\|_{L^1} \lesssim 1,
 \end{equation}
 for alln is an even natural number other than $0$ and $t \in (0,1]$. Finally, the combination of (\ref{equation2.2.2}) and (\ref{equation2.2.3}) imply
 \begin{equation*}
     \left\|\frak{F}^{-1}\left(e^{\lambda_1(|\xi|)t}|\xi|^{s}\chi_{\rm L}(|\xi|)\right)(\chi_{\rm M}(|x|)+\chi_{\rm H}(|x|))\right\|_{L^1} \lesssim 1\quad  \text{ for all }t \in (0,1].\\
 \end{equation*}
 \item[$\bullet$] \textbf{Now we consider the case $t \geq 1$}. We perform the separation technique as follows:
 \begin{align*}
     &\left\|\frak{F}^{-1}\left(e^{\lambda_1(|\xi|)t}|\xi|^s\chi_{\rm L}(|\xi|)\right)\left(\chi_{\rm M}(|x|)+\chi_{\rm H}(|x|)\right)\right\|_{L^1}\\
     &\qquad \leq \left\|\frak{F}^{-1}\left(e^{\lambda_1(|\xi|)t}|\xi|^s\chi_{\rm L}(|\xi|)\right)\left(\chi_{\rm M}(|x|)+\chi_{\rm H}(|x|)\right)\chi_{\rm L}\left(|x|t^{\frac{1}{2(\sigma-\sigma_1)}}\right)\right\|_{L^1}\\
     &\qquad\quad+ \left\|\frak{F}^{-1}\left(e^{\lambda_1(|\xi|)t}|\xi|^s\chi_{\rm L}(|\xi|)\right)\left(\chi_{\rm M}(|x|)+\chi_{\rm H}(|x|)\right)\left(\chi_{\rm M}\left(|x|t^{\frac{1}{2(\sigma-\sigma_1)}}\right)+\chi_{\rm H}\left(|x|t^{\frac{1}{2(\sigma-\sigma_1)}}\right)\right)\right\|_{L^1}.
 \end{align*}
  By the change of variables $\xi = t^{-\frac{1}{2(\sigma-\sigma_1)}}\zeta$ and $y = x t^{\frac{1}{2(\sigma-\sigma_1)}}$ then performing the same proof steps as Lemma \ref{Lemma2.1}, we get the estimate (\ref{estimate2.2.100}).

\end{itemize}
Summarizing, the proof to Lemma \ref{Lemma2.2} is completed.
\end{proof}
\begin{remark}\label{remark1}
    Since the real part of $\lambda_{1,2}$ is negative, we can perform the same proof steps as Lemma \ref{Lemma2.1}, we also get the following estimate:
    \begin{align*}
        \left\|\frak{F}^{-1}\left(|\xi|^{s}e^{\lambda_1(|\xi|) t}\chi_{\rm M}(|\xi|)\right)(t,\cdot)\right\|_{L^1} \lesssim e^{-ct},
    \end{align*}
    for all $t > 0$ and $c$ is suitable positive constant. 
\end{remark} 
Then, from the statements of Lemma \ref{Lemma2.1} and \ref{Lemma2.2}, we may conclude the following $L^1$ estimates.
\begin{lemma}\label{lemma2.3}
The following estimates hold: 
\begin{align*}
    \left\|\frak{F}^{-1}\left(|\xi|^{s}e^{\lambda_1(|\xi|) t}\right)(t,\cdot)\right\|_{L^1} \lesssim 
    \begin{cases}
        t^{-\frac{s}{2(\sigma-\sigma_2)}} &\text{ if } t \in (0,1],\\
        t^{-\frac{s}{2(\sigma-\sigma_1)}} &\text{ if } t \in [1, \infty).
    \end{cases}
\end{align*}

\end{lemma}

\begin{lemma}\label{lemma2.4}
    The following estimate holds:
    \begin{align*}
        \left\|\frak{F}^{-1}\left(e^{(\lambda_2-\lambda_1)(|\xi|)t}\right)(t,\cdot)\right\|_{L^1} \lesssim 1\quad   \text{ for all } t > 0.
    \end{align*}
\end{lemma}
\begin{proof}
   We first consider the case $|\xi| \geq b $. Notice that $\lambda_2 - \lambda_1 \sim -|\xi|^{2\sigma_2}$ for large $|\xi|$. For this reason, perform the same proof steps in Lemma \ref{Lemma2.1}  we have the following conclusion for all $t > 0$:
   \begin{equation*}
       \left\|\frak{F}^{-1}\left(e^{(\lambda_2-\lambda_1)(|\xi|)t}(\chi_{\rm M}(|\xi|)+\chi_{\rm H}(|\xi|))\right)\right\|_{L^1} \lesssim 1.
   \end{equation*}
   Next we consider the case $|\xi| \leq a$. Observing that $\lambda_2 - \lambda_1 \sim -|\xi|^{2\sigma_1}$ for small $|\xi|$ we repeat the same proof steps in Lemma \ref{Lemma2.1} to conclude the following estimate for all $t > 0$:
    \begin{equation*}
       \left\|\frak{F}^{-1}\left(e^{(\lambda_2-\lambda_1)(|\xi|)t}\chi_{\rm L}(|\xi|)\right)\right\|_{L^1} \lesssim 1.
   \end{equation*}
   Therefore, the proof to Lemma \ref{lemma2.4} is completed.
\end{proof}
\begin{proposition}\label{Pro2.3}
    The following estimates hold in $\mathbb{R}^n$ for any $n \geq 1$:
    \begin{align}
        \left\|\frak{F}^{-1}\left(|\xi|^{s}\widehat{K_0}\right)(t,\cdot)\right\|_{L^1} &\lesssim
        \begin{cases}
            t^{-\frac{s}{2(\sigma-\sigma_2)}} &\text{ if } t \in (0,1],\\
            t^{-\frac{s}{2(\sigma-\sigma_1)}} &\text{ if } t \in [1, \infty),
        \end{cases} \label{estimates2.5.1}\\
        \left\|\frak{F}^{-1}\left(|\xi|^{s}\widehat{K_1}\right)(t,\cdot)\right\|_{L^1} &\lesssim
        \begin{cases}
            t^{-\frac{s}{2(\sigma-\sigma_2)}+1} &\text{ if } t \in (0,1],\\
            t^{-\frac{s}{2(\sigma-\sigma_1)}+1} &\text{ if } t \in [1,\infty),
        \end{cases}\label{estimates2.5.2} \\
        \left\|\frak{F}^{-1}\left(|\xi|^{s}\partial_t\widehat{K_1}\right)(t,\cdot)\right\|_{L^1} 
        &\lesssim 
        \begin{cases}
            t^{1-\frac{s+2\sigma_2}{2(\sigma-\sigma_2)}} &\text{ if } t \in (0,1],\\
            t^{1-\frac{s+2\sigma_1}{2(\sigma-\sigma_1)}} &\text{ if } t \in [1, \infty) \label{estimates2.5.3}
            \end{cases}\\
            \left\|\frak{F}^{-1}\left(|\xi|^{s}\partial_t\widehat{K_0}\right)(t,\cdot)\right\|_{L^1} &\lesssim 
        \begin{cases}
            t^{-\frac{s+2\sigma}{2(\sigma-\sigma_2)}+1} &\text{ if } t \in (0,1],\\
            t^{-\frac{s+2\sigma}{2(\sigma-\sigma_1)}+1} &\text{ if } t \in [1, \infty), \label{estimates2.5.4}
        \end{cases}
    \end{align}
    for any non-negative number $s$.
\end{proposition}
    \begin{proof}
        First,taking account of representation for $\widehat{K_1}$ we can re-write it as follows:
        \begin{align}
            \widehat{K_1}(t,\xi) &= e^{\lambda_1 t} \frac{e^{(\lambda_2-\lambda_1)t }-1}{\lambda_2-\lambda_1}= te^{\lambda_1t} \int_0^1 e^{\theta t (\lambda_2-\lambda_1)} d\theta. \label{2.5.5}
        \end{align}
        for small $|\xi|$ and large $|\xi|$. For this reason combine Lemma \ref{lemma2.3} and Lemma \ref{lemma2.4} we have the conclusion
        \begin{align*}
             \left\|\frak{F}^{-1}\left(|\xi|^{s}\widehat{K_1}\right)(t,\cdot)\right\|_{L^1} &\lesssim t \int_0^1 \left\|\frak{F}^{-1}\left(|\xi|^{s}\widehat{K_1}\right)\right\|_{L^1} \left\|\frak{F}^{-1}\left(e^{\theta t (\lambda_2-\lambda_1)}\right)\right\|_{L^1} d\theta \\
             &\lesssim 
        \begin{cases}
            t^{-\frac{s}{2(\sigma-\sigma_2)}+1} &\text{ if } t \in (0,1],\\
            t^{-\frac{s}{2(\sigma-\sigma_1)}+1} &\text{ if } t \in [1,\infty).
        \end{cases}
        \end{align*}
    Next, from the approximations (\ref{approx1}) and (\ref{approx2}) we have
    \begin{align*}
        \left\|\frak{F}^{-1}\left(|\xi|^{s}\widehat{K_0}\right)\right\| &\leq \left\|\frak{F}^{-1}\left(|\xi|^{s}\lambda_1\widehat{K_1}\right)\right\|_{L^1} + \left\|\frak{F}^{-1}\left(|\xi|^{s}e^{\lambda_1 t}\right)\right\|_{L^1}\\
        &\lesssim 
        \begin{cases}
            t^{-\frac{s}{2(\sigma-\sigma_2)}} &\text{ if } t \in (0,1],\\
            t^{-\frac{s}{2(\sigma-\sigma_1)}} &\text{ if } t \in [1, \infty).
        \end{cases} 
    \end{align*}
   Finally, taking account of some estimates related to the time derivative, we note that
   \begin{align}
       \partial_t \widehat{K_0} = -|\xi|^{2\sigma} \widehat{K_1} \text{ and } \partial_t\widehat{K_1} = \widehat{K_0} - \left(|\xi|^{2\sigma_1}+|\xi|^{2\sigma_2}\right) \widehat{K_1}. \label{equation2.5.1}
   \end{align}
   From there we have estimates (\ref{estimates2.5.3}) and (\ref{estimates2.5.4}).  Summarizing, the proof to Proposition \ref{Pro2.3} is completed.
    \end{proof}
%\end{lemma}

\subsection{$L^\ity$ estimates}
\begin{lemma}\label{lemma2.5}
    The following estimates hold in $\mathbb{R}^n$ for any $n \geq 1$:
    \begin{align*}
        \left\|\frak{F}^{-1}\left(|\xi|^{s}\widehat{K_1}\chi_{\rm L}(|\xi|)\right)\right\|_{L^{\infty}} &\lesssim
        \begin{cases}
            t &\text{ if } t \in (0,1],\\
            t^{-\frac{n+s}{2(\sigma-\sigma_1)}+1} &\text{ if } t \in [1,\infty),
        \end{cases}\\
        \left\|\frak{F}^{-1}\left(|\xi|^{s}\widehat{K_1}(\chi_{\rm M}(|\xi|)+\chi_{\rm H}(|\xi|))\right)\right\|_{L^{\infty}} &\lesssim t^{1-\frac{s+n}{2(\sigma-\sigma_2)}} \quad\text{ for any } t 
        \in (0, \infty), \\
        \left\|\frak{F}^{-1}\left(|\xi|^{s}\widehat{K_0}\chi_{\rm L}(|\xi|)\right)\right\|_{L^{\infty}} &\lesssim
        \begin{cases}
            1 &\text{ if } t \in (0,1],\\
            t^{-\frac{n+s}{2(\sigma-\sigma_1)}} &\text{ if } t \in [1,\infty),
        \end{cases}\\
    \left\|\frak{F}^{-1}\left(|\xi|^{s}\widehat{K_0}(\chi_{\rm M}(|\xi|)+\chi_{\rm H}(|\xi|))\right)\right\|_{L^{\infty}} &\lesssim t^{-\frac{n+s}{2(\sigma-\sigma_2)}} \quad\text{ for any } t \in (0, \infty),
    \end{align*}
    for any non-negative number $s$.
\end{lemma}
\begin{proof}
    For the sake of the asymptotic behavior of the characteristic roots in (\ref{approx1}) and (\ref{approx2}), combined with performance (\ref{2.5.5}) we have the following conclusion:
    \begin{align*}
        \left|\widehat{K_1}(t,\xi)\right| &\lesssim t e^{-|\xi|^{2(\sigma-\sigma_1)}t}, \quad \left|\widehat{K_0}(t,\xi)\right| \lesssim e^{-c_1|\xi|^{2(\sigma-\sigma_1)}t} \text{ for small } |\xi|,\\
        \left|\widehat{K_1}(t,\xi)\right| &\lesssim te^{-|\xi|^{2(\sigma-\sigma_2)}t}, \quad \left|\widehat{K_0}(t,\xi)\right| \lesssim e^{-|\xi|^{2(\sigma-\sigma_2)}t} \text{ for large } |\xi|, 
    \end{align*}
    where $c$ is a suitable positive constant. Therefore, the proof to Lemma \ref{lemma2.5} is completed. 
\end{proof}
From  relationship (\ref{equation2.5.1}) and Lemma \ref{lemma2.5} the following statement follows immediately.
\begin{proposition}\label{Pro2.4}
    The following estimates hold in $\mathbb{R}^n$ for any $n \geq 1$:
    \begin{align*}
        \left\|\frak{F}^{-1}\left(|\xi|^{s}\widehat{K_1}\right)\right\|_{L^{\infty}} &\lesssim 
        \begin{cases}
            t^{-\frac{n}{2(\sigma-\sigma_2)}-\frac{s}{2(\sigma-\sigma_2)}+1} &\text{ if } t \in (0,1],\\
            t^{-\frac{n}{2(\sigma-\sigma_1)}-\frac{s}{2(\sigma-\sigma_1)}+1} &\text{ if } t \in [1,\infty),
        \end{cases}\\
        \left\|\frak{F}^{-1}\left(|\xi|^{s}\widehat{K_0}\right)\right\|_{L^{\infty}} &\lesssim 
        \begin{cases}
             t^{-\frac{n}{2(\sigma-\sigma_2)}-\frac{s}{2(\sigma-\sigma_2)}} &\text{ if } t \in (0,1],\\
            t^{-\frac{n}{2(\sigma-\sigma_1)}-\frac{s}{2(\sigma-\sigma_1)}} &\text{ if } t \in [1,\infty),
        \end{cases}\\
        \left\|\frak{F}^{-1}\left(|\xi|^{s}\partial_t\widehat{K_1}\right)\right\|_{L^{\infty}} &\lesssim 
        \begin{cases}
            t^{-\frac{n}{2(\sigma-\sigma_2)}-\frac{s+2\sigma_2}{2(\sigma-\sigma_2)}+1} &\text{ if } t \in (0,1],\\
            t^{-\frac{n}{2(\sigma-\sigma_1)}-\frac{s+2\sigma_1}{2(\sigma-\sigma_1)}+1} &\text{ if } t \in [1, \infty),
        \end{cases}\\
        \left\|\frak{F}^{-1}\left(|\xi|^{s}\partial_t\widehat{K_0}\right)\right\|_{L^{\infty}} &\lesssim 
        \begin{cases}
            t^{-\frac{n}{2(\sigma-\sigma_2)}-\frac{s+2\sigma}{2(\sigma-\sigma_2)}+1} &\text{ if } t \in (0,1],\\
            t^{-\frac{n}{2(\sigma-\sigma_1)}-\frac{s+2\sigma}{2(\sigma-\sigma_1)}+1} &\text{ if } t \in [1,\infty),
        \end{cases}
    \end{align*}
    for any non-negative number $s$.
\end{proposition}
\subsection{$L^r$ estimates} 
From the statements of Proposition \ref{Pro2.3} and Proposition \ref{Pro2.4}, by applying Littlewood's inequality we may conclude the following $L^r$ estimates.

\begin{proposition}
   The following estimates hold in $\mathbb{R}^n$ for any $n \geq 1$:
   \begin{align*}
       \left\|\frak{F}^{-1}\left(|\xi|^{s}\widehat{K_0}\right)(t,\cdot)\right\|_{L^r} &\lesssim
       \begin{cases}
           t^{-\frac{n}{2(\sigma-\sigma_2)}\left(1-\frac{1}{r}\right)-\frac{s}{2(\sigma-\sigma_2)}} &\text{ if } t \in (0,1],\\
           t^{-\frac{n}{2(\sigma-\sigma_1)}\left(1-\frac{1}{r}\right)-\frac{s}{2(\sigma-\sigma_1)}} &\text{ if } t \in [1,\infty),
           \end{cases}\\
       \left\|\frak{F}^{-1}\left(|\xi|^{s}\widehat{K_1}\right)(t,\cdot)\right\|_{L^r} &\lesssim
       \begin{cases}
           t^{-\frac{n}{2(\sigma-\sigma_2)}\left(1-\frac{1}{r}\right)-\frac{s}{2(\sigma-\sigma_2)}+1} &\text{ if } t \in (0,1],\\
           t^{-\frac{n}{2(\sigma-\sigma_1)}\left(1-\frac{1}{r}\right)-\frac{s}{2(\sigma-\sigma_1)}+1} &\text{ if } t \in [1,\infty),
       \end{cases}
       \end{align*}
       \begin{align*}
       \left\|\frak{F}^{-1}\left(|\xi|^{s}\partial_t\widehat{K_0}\right)(t,\cdot)\right\|_{L^r} &\lesssim
       \begin{cases}
           t^{-\frac{n}{2(\sigma-\sigma_2)}\left(1-\frac{1}{r}\right)-\frac{s+2\sigma}{2(\sigma-\sigma_2)}+1} &\text{ if } t \in (0,1],\\
           t^{-\frac{n}{2(\sigma-\sigma_1)}\left(1-\frac{1}{r}\right)-\frac{s+2\sigma}{2(\sigma-\sigma_1)}+1} &\text{ if } t \in [1,\infty),
       \end{cases}\\
       \left\|\frak{F}^{-1}\left(|\xi|^{s}\partial_t\widehat{K_1}\right)(t,\cdot)\right\|_{L^r} &\lesssim
       \begin{cases}
           t^{-\frac{n}{2(\sigma-\sigma_2)}\left(1-\frac{1}{r}\right)-\frac{s+2\sigma_2}{2(\sigma-\sigma_2)}+1} &\text{ if } t \in (0,1],\\
           t^{-\frac{n}{2(\sigma-\sigma_1)}\left(1-\frac{1}{r}\right)-\frac{s+2\sigma_1}{2(\sigma-\sigma_1)}+1} &\text{ if } t \in [1,\infty),
       \end{cases}
   \end{align*}
   for all $r \in [1, \infty]$ and any non-negative $s$.
\end{proposition}
\subsection{$L^p-L^q$ estimates}
\begin{proposition}[$L^p-L^q$ estimates not necessarily on the conjugate line]
    Let $0 < \sigma_1 < \frac{\sigma}{2} < \sigma_2 < \sigma$ in (\ref{Main.Eq.3}) and $1 \leq m < q < \infty$. Then, the solutions to (\ref{Main.Eq.3}) satisfy the following $L^p-L^q$ estimates:
    \begin{align*}
        \left\||D|^{s}u(t,\cdot)\right\|_{L^q} &\lesssim 
        \begin{cases}
            t^{-\frac{n}{2(\sigma-\sigma_2)}\left(1-\frac{1}{r}\right)-\frac{s}{2(\sigma-\sigma_2)}} \|u_0\|_{L^p}+   t^{-\frac{n}{2(\sigma-\sigma_2)}\left(1-\frac{1}{r}\right)-\frac{s}{2(\sigma-\sigma_2)}+1} \|u_1\|_{L^p} &\text{ if } t \in (0,1],\\
            t^{-\frac{n}{2(\sigma-\sigma_1)}\left(1-\frac{1}{r}\right)-\frac{s}{2(\sigma-\sigma_1)}} \|u_0\|_{L^p} + t^{-\frac{n}{2(\sigma-\sigma_1)}\left(1-\frac{1}{r}\right)-\frac{s}{2(\sigma-\sigma_1)}+1} \|u_1\|_{L^p} &\text{ if } t \in [1,\infty),
        \end{cases}\\
       \||D|^{s}u_t(t,\cdot)\|_{L^q} &\lesssim
       \begin{cases}
           t^{-\frac{n}{2(\sigma-\sigma_2)}\left(1-\frac{1}{r}\right)-\frac{s+2\sigma_2}{2(\sigma-\sigma_2)}} \|u_0\|_{L^p}+   t^{-\frac{n}{2(\sigma-\sigma_2)}\left(1-\frac{1}{r}\right)-\frac{s+2\sigma_2}{2(\sigma-\sigma_2)}+1} \|u_1\|_{L^p} &\text{ if } t \in (0,1],\\
            t^{-\frac{n}{2(\sigma-\sigma_1)}\left(1-\frac{1}{r}\right)-\frac{s+2\sigma_1}{2(\sigma-\sigma_1)}} \|u_0\|_{L^p} + t^{-\frac{n}{2(\sigma-\sigma_1)}\left(1-\frac{1}{r}\right)-\frac{s+2\sigma_1}{2(\sigma-\sigma_1)}+1} \|u_1\|_{L^p} &\text{ if } t \in [1,\infty),
       \end{cases}
    \end{align*}
    where $1+\frac{1}{q} = \frac{1}{r} + \frac{1}{p}$, for any non-negative number $a$ and for all $n \geq 1$.
\end{proposition}

\begin{proposition}[$L^q-L^q$ linear estimates with additional $L^m$ regularity for the data]\label{theorem2.2}
     Let $0 < \sigma_1 < \frac{\sigma}{2} < \sigma_2 < \sigma$ in (\ref{Main.Eq.3}) and $1 \leq m < q < \infty$. Then, the solutions to (\ref{Main.Eq.3}) satisfy the following $(L^m \cap L^q)-L^q$ estimates:
     \begin{align*}
         \left\|u(t,\cdot)\right\|_{L^q} &\lesssim
         \begin{cases}
             \|u_0\|_{L^m \cap L^q} + t\|u_1\|_{L^m \cap L^q} &\text{ if } t \in (0,1],\\
             (1+t)^{-\frac{n}{2(\sigma-\sigma_1)}\left(1-\frac{1}{r}\right)}\|u_0\|_{L^m \cap L^q} \\
             \qquad \quad+ (1+t)^{-\frac{n}{2(\sigma-\sigma_1)}\left(1-\frac{1}{r}\right)+1}\|u_1\|_{L^m \cap L^q} &\text{ if } t \in [1, \infty),
         \end{cases}\\
         \left\||D|^{s}u(t,\cdot)\right\|_{L^q} &\lesssim
         \begin{cases}
             \|u_0\|_{L^m \cap H^s_q} + t^{1-\frac{s}{2(\sigma-\sigma_2)}}\|u_1\|_{L^m \cap H^{[s-2\sigma_2]^{+}}_q} &\text{ if } t \in (0,1],\\
             (1+t)^{-\frac{n}{2(\sigma-\sigma_1)}\left(1-\frac{1}{r}\right)-\frac{s}{2(\sigma-\sigma_1)}}\|u_0\|_{L^m \cap H^s_q} \\
             \qquad \quad+ (1+t)^{-\frac{n}{2(\sigma-\sigma_1)}\left(1-\frac{1}{r}\right)-\frac{s}{2(\sigma-\sigma_1)}+1}\|u_1\|_{L^m \cap H^{[s-2\sigma_2]^{+}}_q} &\text{ if } t \in [1, \infty),
         \end{cases}\\
       \left\|u_t(t,\cdot)\right\|_{L^q} &\lesssim 
       \begin{cases}
          \|u_0\|_{L^m \cap H^{2\sigma_2}_q}+ t^{1-\frac{\sigma_2}{\sigma-\sigma_2}}\|u_1\|_{L^m \cap L^q} &\text{ if } t \in (0,1],\\
           (1+t)^{-\frac{n}{2(\sigma-\sigma_1)}\left(1-\frac{1}{r}\right)-\frac{\sigma_1}{\sigma-\sigma_1}}\|u_0\|_{L^m \cap L^q} \\
           \qquad \quad+ (1+t)^{-\frac{n}{2(\sigma-\sigma_1)}\left(1-\frac{1}{r}\right)-\frac{\sigma_1}{\sigma-\sigma_1}+1}\|u_1\|_{L^m \cap L^q} &\text{ if } t \in [1, \infty),
       \end{cases} \\
       \left\||D|^{s}u_t(t,\cdot)\right\|_{L^q} &\lesssim 
       \begin{cases}
          \|u_0\|_{L^m \cap H^{s+2\sigma_2}_q}+ t^{1-\frac{\sigma_2}{\sigma-\sigma_2}}\|u_1\|_{L^m \cap H^{s}_q} &\text{ if } t \in (0,1],\\
           (1+t)^{-\frac{n}{2(\sigma-\sigma_1)}\left(1-\frac{1}{r}\right)-\frac{s+2\sigma_1}{2(\sigma-\sigma_1)}}\|u_0\|_{L^m \cap H^{s+2\sigma_2}_q} \\
           \qquad \quad+ (1+t)^{-\frac{n}{2(\sigma-\sigma_1)}\left(1-\frac{1}{r}\right)-\frac{s+2\sigma_1}{2(\sigma-\sigma_1)}+1}\|u_1\|_{L^m \cap H^{s}_q} &\text{ if } t \in [1, \infty),
       \end{cases}
     \end{align*}
     and the $L^q-L^q$ estimates
     \begin{align*}
         \left\||D|^{s}u(t,\cdot)\right\|_{L^q} &\lesssim
         \begin{cases}
             \|u_0\|_{ H^s_q} +t^{1-\frac{s}{2(\sigma-\sigma_2)}} \|u_1\|_{ H^{[s-2\sigma_2]^{+}}_q} &\text{ if } t \in (0,1],\\
             (1+t)^{-\frac{s}{2(\sigma-\sigma_1)}}\|u_0\|_{H^s_q} \\
             \qquad \quad+ (1+t)^{1-\frac{s}{2(\sigma-\sigma_1)}}\|u_1\|_{H^{[s-2\sigma_2]^{+}}_q} &\text{ if } t \in [1, \infty),
         \end{cases}\\
          \left\||D|^{s}u_t(t,\cdot)\right\|_{L^q} &\lesssim 
       \begin{cases}
           \|u_0\|_{ H^{s+2\sigma_2}_q}+t^{1-\frac{\sigma_2}{\sigma-\sigma_2}}\|u_1\|_{ H^{s}_q} &\text{ if } t \in (0,1],\\
           (1+t)^{-\frac{s+2\sigma_1}{2(\sigma-\sigma_1)}}\|u_0\|_{ H^{s+2\sigma_2}_q} \\
           \qquad \quad+ (1+t)^{1-\frac{s+2\sigma_1}{2(\sigma-\sigma_1)}}\|u_1\|_{ H^{s}_q} &\text{ if } t \in [1, \infty),
       \end{cases}
     \end{align*}
    where $1+\frac{1}{q} = \frac{1}{r} + \frac{1}{m}$, for any non-negative number $s$ and for all $n \geq 1$.
\end{proposition}
\begin{proof}
    First, from relationship (\ref{2.5.5}) and Lemma \ref{Lemma2.2} we can prove the following estimates
    \begin{align*}
\left\|\frak{F}^{-1}\left(|\xi|^{s}\widehat{K_0} \chi_{\rm L}(|\xi|)\right)\right\|_{L^r} &\lesssim
        \begin{cases}
            1 &\text{ if } t \in (0,1],\\
            t^{-\frac{n}{2(\sigma-\sigma_1)}\left(1-\frac{1}{r}\right)-\frac{s}{2(\sigma-\sigma_1)}} &\text{ if } t \in [1, \infty),\\
        \end{cases}\\
        \left\|\frak{F}^{-1}\left(|\xi|^{s}\widehat{K_1} \chi_{\rm L}(|\xi|)\right)\right\|_{L^r} &\lesssim
        \begin{cases}
            t &\text{ if } t \in (0,1],\\
            t^{-\frac{n}{2(\sigma-\sigma_1)}\left(1-\frac{1}{r}\right)-\frac{s}{2(\sigma-\sigma_1)}+1} &\text{ if } t \in [1, \infty),
        \end{cases}
    \end{align*}
    for all $r \in [1, \infty]$ any non-negative number $s$. Therefore, applying Young's convolution inequality and using a suitable regularity of the data $u_0$ and $u_1$ depending on the order of $s$, we may conclude all the desired estimates for the solution and some of its derivaties.
\end{proof}
\begin{corollary}\label{corollary2.2}
    Let $0 < \sigma_1 < \frac{\sigma}{2} < \sigma_2 < \sigma$ in (\ref{Main.Eq.3}) and $1 \leq m < q < \infty$. Then, the solutions to (\ref{Main.Eq.3}) satisfy the following $(L^m \cap L^q)-L^q$ estimates:
     \begin{align*}
         \left\|u(t,\cdot)\right\|_{L^q} &\lesssim
         \begin{cases}
             \|u_0\|_{L^m \cap L^q} + t\|u_1\|_{L^m \cap L^q} &\text{ if } t \in (0,1],\\
             (1+t)^{-\frac{n}{2(\sigma-\sigma_1)}\left(\frac{1}{m}-\frac{1}{q}\right)}\|u_0\|_{L^m \cap L^q} \\
             \qquad \quad+ (1+t)^{-\frac{n}{2(\sigma-\sigma_1)}\left(\frac{1}{m}-\frac{1}{q}\right)+1}\|u_1\|_{L^m \cap L^q} &\text{ if } t \in [1, \infty),
         \end{cases}\\
         \left\||D|^{s}u(t,\cdot)\right\|_{L^q} &\lesssim
         \begin{cases}
             \|u_0\|_{L^m \cap H^s_q} + t^{\frac{1}{2}}\|u_1\|_{L^m \cap H^{[s+\sigma_2-\sigma]^{+}}_q} &\text{ if } t \in (0,1],\\
             (1+t)^{-\frac{n}{2(\sigma-\sigma_1)}\left(\frac{1}{m}-\frac{1}{q}\right)-\frac{s}{2(\sigma-\sigma_1)}}\|u_0\|_{L^m \cap H^s_q} \\
             \qquad \quad+ (1+t)^{-\frac{n}{2(\sigma-\sigma_1)}\left(\frac{1}{m}-\frac{1}{q}\right)-\frac{s}{2(\sigma-\sigma_1)}+1}\|u_1\|_{L^m \cap H^{[s-2\sigma_2]^{+}}_q} &\text{ if } t \in [1, \infty),
         \end{cases}\\
       \left\|u_t(t,\cdot)\right\|_{L^q} &\lesssim 
       \begin{cases}
          \|u_0\|_{L^m \cap H^{2\sigma_2}_q}+ t^{\frac{1}{2}}\|u_1\|_{L^m \cap H^{3\sigma_2-\sigma}_q} &\text{ if } t \in (0,1],\\
           (1+t)^{-\frac{n}{2(\sigma-\sigma_1)}\left(\frac{1}{m}-\frac{1}{q}\right)-\frac{\sigma_1}{\sigma-\sigma_1}}\|u_0\|_{L^m \cap L^q} \\
           \qquad \quad+ (1+t)^{-\frac{n}{2(\sigma-\sigma_1)}\left(\frac{1}{m}-\frac{1}{q}\right)-\frac{\sigma_1}{\sigma-\sigma_1}+1}\|u_1\|_{L^m \cap L^q} &\text{ if } t \in [1, \infty),
       \end{cases}\\
       \left\||D|^{s}u_t(t,\cdot)\right\|_{L^q} &\lesssim 
       \begin{cases}
          \|u_0\|_{L^m \cap H^{s+2\sigma_2}_q}+ t^{\frac{1}{2}}\|u_1\|_{L^m \cap H^{s+3\sigma_2-\sigma}_q} &\text{ if } t \in (0,1],\\
           (1+t)^{-\frac{n}{2(\sigma-\sigma_1)}\left(\frac{1}{m}-\frac{1}{q}\right)-\frac{s+2\sigma_1}{2(\sigma-\sigma_1)}}\|u_0\|_{L^m \cap H^{s+2\sigma_2}_q} \\
           \qquad \quad+ (1+t)^{-\frac{n}{2(\sigma-\sigma_1)}\left(\frac{1}{m}-\frac{1}{q}\right)-\frac{s+2\sigma_1}{2(\sigma-\sigma_1)}+1}\|u_1\|_{L^m \cap H^{s}_q} &\text{ if } t \in [1, \infty),
       \end{cases}
     \end{align*}
     and the $L^q-L^q$ estimates
     \begin{align*}
         \left\||D|^{s}u(t,\cdot)\right\|_{L^q} &\lesssim
         \begin{cases}
             \|u_0\|_{ H^s_q} +t^{\frac{1}{2}}\|u_1\|_{ H^{[s +\sigma_2-\sigma]^{+}}_q} &\text{ if } t \in (0,1]\\
             (1+t)^{-\frac{s}{2(\sigma-\sigma_1)}}\|u_0\|_{H^s_q} \\
             \qquad \quad+ (1+t)^{1-\frac{s}{2(\sigma-\sigma_1)}}\|u_1\|_{H^{[s-2\sigma_2]^{+}}_q} &\text{ if } t \in [1, \infty),
         \end{cases}\\
          \left\||D|^{s}u_t(t,\cdot)\right\|_{L^q} &\lesssim 
       \begin{cases}
           \|u_0\|_{ H^{s+2\sigma_2}_q}+ t^{\frac{1}{2}}\|u_1\|_{ H^{s +3\sigma_2-\sigma}_q} &\text{ if } t \in (0,1],\\
           (1+t)^{-\frac{s+2\sigma_1}{2(\sigma-\sigma_1)}}\|u_0\|_{ H^{s+2\sigma_2}_q} \\
           \qquad \quad+ (1+t)^{1-\frac{s+2\sigma_1}{2(\sigma-\sigma_1)}}\|u_1\|_{ H^{s}_q} &\text{ if } t \in [1, \infty),
       \end{cases}
     \end{align*}
 for any non-negative number $s$ and for all $n \geq 1$.
\end{corollary}

\section{Global (in time) existence of small data solutions}\label{Sec.3}
\subsection{Auxiliary estimates}
In this section, we will use the decay estimates for $(\ref{Main.Eq.3})$, which are obtained in Proposition \ref{theorem2.2} and Corollary \ref{corollary2.2} to prove the corresponding Theorems \ref{theorem1.0}-\ref{theorem1.3}. Our main tools are Duhamel’s principle and Gagliardo–Nirenberg inequality. By using fundamental solutions we write the solution of (\ref{Main.Eq.3}) in the form 
$$
u(t,x) = K_0(t,x) * u_0(x) + K_1(t,x) * u_1(x),
$$
then the solution to $(\ref{Main.Eq.3})$ becomes
\begin{equation}\label{2.31}
    u(t,x) = K_0(t,x) \ast_x u_0(x) + K_1(t,x) \ast_x u_1(x) + \int_0^t K_1(t-\tau,x) \ast_x |u(\tau,x)|^p d\tau.
\end{equation}
We consider the space of data  $\frak{D}_{m,q}^{s} := (L^m \cap H^{s}_q) \times (L^m \cap H^{[s+\sigma_2-\sigma]^{+}}_q)$ , $m \in [1,q)$ and the function spaces $X(t) := C([0,t], H^s_q) \cap C^1([0,t], H^{[s-2\sigma_2]^{+}}_q) $ for all $t > 0$ with the norm
\begin{align*}
    \|u\|_{X(t)}:=
    \sup _{0 \leq \tau \leq t}&\big(f_1(\tau)^{-1}\|u(\tau, \cdot)\|_{L^q}+f_{2,s}(\tau)^{-1}\left\||D|^{s} u(\tau, \cdot)\right\|_{L^q}\\
&\qquad \quad+f_3(\tau)^{-1}\left\|u_t(\tau, \cdot)\right\|_{L^q}+f_{4,s}(\tau)^{-1} \||D|^{s-2\sigma_2}u_t(\tau,\cdot)\|_{L^q}\big),
\end{align*}
where from the estimates of Corollary \ref{corollary2.2} we choose 
\begin{align*}
    f_1(\tau) &:= (1+\tau)^{1-\frac{n}{2(\sigma-\sigma_1)}(\frac{1}{m}-\frac{1}{q})}, \\
    f_{2,s}(\tau) &:=  (1+\tau)^{1-\frac{n}{2(\sigma-\sigma_1)}(\frac{1}{m}-\frac{1}{q}) - \frac{s}{2(\sigma-\sigma_1)}},\\
    f_3(\tau) &:= (1+\tau)^{1-\frac{n}{2(\sigma-\sigma_1)}(\frac{1}{m}-\frac{1}{q})-\frac{\sigma_1}{\sigma-\sigma_1}},\\
    f_{4,s}(\tau) &:=(1+\tau)^{1-\frac{n}{2(\sigma-\sigma_1)}(\frac{1}{m}-\frac{1}{q})-\frac{s}{2(\sigma-\sigma_1)}+\frac{\sigma_2-\sigma_1}{\sigma-\sigma_1}}.
\end{align*}
We define for any $u \in X(t)$ the operator 
\begin{align*}
    N \text{ : } &u \in X(t) \longrightarrow Nu \in X(t)  \\
    Nu(t,x) &= K_0(t,x) \ast_x u_0(x) + K_1(t,x) \ast_x u_1(x) + \int_0^t K_1(t-\tau,x) \ast_x |u(\tau,x)|^p d\tau \\
   &= u^{\rm lin}(t,x) + u^{\rm non}(t,x).
\end{align*}
The mapping into $X(t)$ follows from estimate 
\begin{equation}\label{estimates3.1.1}
    \|Nu(t,\cdot)\|_{X(t)} \lesssim \|(u_0,u_1)\|_{\frak{D}_{m,q}^s} + \|u(t,\cdot)\|_{X(t)}^p.
\end{equation}
Moreover, we show the Lipschitz property 
\begin{equation}\label{estimates3.1.2}
    \|Nu(t,\cdot)-Nv(t,\cdot)\|_{X(t)} \lesssim \|u(t,\cdot)-v(t,\cdot)\|_{X(t)}\left(\|u(t,\cdot)\|_{X(t)}^{p-1} + \|v(t,\cdot)\|_{X(t)}^{p-1}\right).
\end{equation}
If we are able to prove $(\ref{estimates3.1.1})$ and $(\ref{estimates3.1.2})$, then standard arguments imply on the one hand local in time solutions
for arbitrary data and on the other hand global in time solutions for small data as well(see [1]). The decay
estimates for the solution and its energy follow immediately with the definition of the norm in X(t).
\begin{remark}
    From the definition of the norm in $X(t)$ and the statements of 
    Corollary \ref{corollary2.2} we conclude
\begin{equation}\label{estimates1.1.5}
    \|u^{\rm lin}\|_{X(t)} \lesssim \|(u_0,u_1)\|_{\frak{D}_{m,q}^s}.
\end{equation}
For this reason, to complete the proof of (\ref{estimates3.1.1}) we need to show the following inequality:
\begin{equation}\label{estimates1.1.6}
    \|u^{\rm non}\|_{X(t)} \lesssim \|u\|_{X(t)}^p.
\end{equation}
\end{remark}
\subsection{Proof of Theorem \ref{theorem1.0}}
We introduction the data space $\frak{D}^{s}_{m,q}$ and the solution space 
\begin{align*}
    X(t) := C([0,t], H^s_q),
\end{align*}
where the weight $f_3(\tau) = f_{4,s}(\tau) = 0$. Let us prove the inequality (\ref{estimates1.1.6}). From Corollary \ref{corollary2.2} we have
 \begin{align*}
    \|u^{\rm non}(t,\cdot)\|_{L^q} \lesssim& \int_0^{[t-1]^{+}} (1+t-\tau)^{1-\frac{n}{2(\sigma-\sigma_1)}(\frac{1}{m}-\frac{1}{q})} \left\||u(\tau, \cdot)|^p\right\|_{L^m \cap L^q} d\tau\\
    &+ \int_{[t-1]^{+}}^t (t-\tau) \left\||u(\tau, \cdot)|^p\right\|_{L^q} d\tau,\\
     \||D|^{s}u^{\rm non}(t,\cdot)\|_{L^q} \lesssim& \int_0^{[t-1]^{+}} (1+t-\tau)^{1-\frac{n}{2(\sigma-\sigma_1)}(\frac{1}{m}-\frac{1}{q})-\frac{s}{2(\sigma-\sigma_1)}}\left\||u(\tau, \cdot)|^p\right\|_{L^m\cap L^q} d\tau\\
    &+ \int_{[t-1]^{+}}^t \left\||u(\tau, \cdot)|^p\right\|_{L^q\cap \dot{H}^{[s-\sigma+\sigma_2]^{+}}_q} d\tau.
\end{align*}
Hence, it is necessary to require the estimates for $|u(\tau,\cdot)|^p$ in $L^m \cap L^q$ and $L^q$ as follows:
\begin{equation*}
    \||u(\tau,\cdot)|^p\|_{L^m \cap L^q} \lesssim \|u(\tau,\cdot)\|_{L^{mp}}^p +\|u(\tau,\cdot)\|_{L^{qp}}^p \text{ and } \||u(\tau,\cdot)|^p\|_{L^q} = \|u(\tau,\cdot)\|_{L^{qp}}^p.
\end{equation*}
Applying the fractional Gagliardo-Nirenberg inequality we can conclude
\begin{align}
    \|u(\tau,\cdot)\|_{L^{mp}}^p \lesssim & \|u(\tau,\cdot)\|_{L^q}^{p(1-\theta_1)}\|u(\tau,\cdot)\|_{\dot{H}^s_q}^{p\theta_1}\notag\\
    \lesssim &(1+\tau)^{p -\frac{n}{2m(\sigma-\sigma_1)}(p-1)} \|u\|_{X(t)}^p, \label{Main.Es1.0.1}\\
    \|u(\tau,\cdot)\|_{L^{qp}}^p
    \lesssim & \|u(\tau,\cdot)\|_{L^q}^{p(1-\theta_2)}\|u(\tau,\cdot)\|_{\dot{H}^s_q}^{p\theta_2}\notag\\
    \lesssim & (1+\tau)^{p-\frac{np}{2(\sigma-\sigma_1)}(\frac{1}{m}-\frac{1}{qp})} \|u\|_{X(t)}^p, \label{Main.Es1.0.2}
\end{align}
where $\theta_1 = \frac{n}{s}\left(\frac{1}{q}-\frac{1}{mp}\right) \in [0,1]$ and $\theta_2 = \frac{n}{s}\left(\frac{1}{q}-\frac{1}{qp}\right) \in [0,1]$. This condition implies
\begin{align*}
    p \in \left[\frac{q}{m}, \infty\right) \text{ if } n \leq qs \text{ or } p \in \left[\frac{q}{m}, \frac{n}{n-qs}\right] \text{ if } n \in \left(qs, \frac{q^2s}{q-m}\right].
\end{align*}
 The condition (\ref{condition1.0.1}) implies
\begin{align}
    p-\frac{n}{2m(\sigma-\sigma_1)}(p-1) &< -1,\label{result1.1.1}\\
    p-\frac{np}{2(\sigma-\sigma_1)}\left(\frac{1}{m}-\frac{1}{qp}\right) &= \left(p-\frac{n}{2m(\sigma-\sigma_1)}(p-1)\right) - \frac{n}{2(\sigma-\sigma_1)}\left(\frac{1}{m}-\frac{1}{q}\right) \notag\\
    &< -1- \frac{n}{2(\sigma-\sigma_1)}\left(\frac{1}{m}-\frac{1}{q}\right).\label{result1.0.2}
\end{align}
 Now, we evaluate the integral containing the term $\||u(\tau,\cdot)|^p\|_{\dot{H}^{[s-\sigma+\sigma_2]^{+}}_q}$. Here we only need to pay attention to the case $s > \sigma-\sigma_2$.
 \begin{itemize}[leftmargin=*]
     \item If $s-\sigma+\sigma_2 \leq \frac{n}{q}$, applying fractional chain rule with $p >  \lceil s-\sigma+\sigma_2 \rceil$ and Gagliardo- Nirenberg inequality we obtain
 \begin{align*}
     \||u(\tau,\cdot)|^p\|_{\dot{H}^{s-\sigma+\sigma_2}_q} \lesssim& \|u(\tau,\cdot)\|_{L^{q_1}}^{p-1}\||D|^{s-\sigma+\sigma_2}u(\tau,\cdot)\|_{L^{q_2}}\notag\\
     \lesssim& \|u(\tau,\cdot)\|_{L^q}^{(p-1)(1-\theta_3)+1-\theta_4} \|u(\tau,\cdot)\|_{\dot{H}^s_q}^{(p-1)\theta_3+\theta_4}\notag\\
     \lesssim& (1+\tau)^{p-\frac{np}{2(\sigma-\sigma_1)}\left(\frac{1}{m}-\frac{1}{qp}\right)-\frac{s-\sigma+\sigma_2}{2(\sigma-\sigma_1)}} \|u\|_{X(t)}^p,
 \end{align*}
 where $\theta_3 = \frac{n}{s}\left(\frac{1}{q}-\frac{1}{q_1}\right) \in [0,1]$ , $\theta_4 = \frac{n}{s}\left(\frac{1}{q}-\frac{1}{q_2}+\frac{s-\sigma+\sigma_2}{n}\right) \in \left[\frac{s-\sigma+\sigma_2}{s}, 1\right]$ and $\frac{p-1}{q_1}+\frac{1}{q_2} = \frac{1}{q}$. This condition implies
 \begin{align*}
     p > 1 \text{ if } n \in [q(s-\sigma+\sigma_2), qs] \text{ or } p \in \left(1, \frac{n-q(s-\sigma+\sigma_2)}{n-qs}\right] \text{ if } n > qs.
 \end{align*}
 \item If $s-\sigma+\sigma_2 > \frac{n}{q}$, applying fractional powers, fractional Sobolev embedding and Gagliardo-Nirenberg inequality we obtain
 \begin{align*}
     \| |u(\tau,\cdot)|^p\|_{\dot{H}^{s-\sigma+\sigma_2}_q} &\lesssim \|u(\tau,\cdot)\|_{L^{\infty}}^{p-1} \|u(\tau,\cdot)\|_{\dot{H}^{s-\sigma+\sigma_2}_q}\\
     &\lesssim \|u(\tau,\cdot)\|_{\dot{H}^{s-\sigma+\sigma_2}_q} \|u(\tau,\cdot)\|_{\dot{H}^{s^{*}}_q}^{p-1}\\
     &\lesssim (1+\tau)^{p(1-\frac{n}{2(\sigma-\sigma_1)}(\frac{1}{m}-\frac{1}{q}))-\frac{s-\sigma+\sigma_2}{2(\sigma-\sigma_1)}-(p-1)\frac{s^{*}}{2(\sigma-\sigma_1)}}\|u\|_{X(t)}^p\\
     &\lesssim (1+\tau)^{p-\frac{np}{2(\sigma-\sigma_1)}(\frac{1}{m}-\frac{1}{qp})}\|u\|_{X(t)}^p,
 \end{align*}
 \end{itemize}
where we choose $s*= \frac{n}{q}-\epsilon$ with $\epsilon$ is a sufficiently small positive number. Summarizing, we get estimate
 \begin{align}
      \| |u(\tau,\cdot)|^p\|_{\dot{H}^{s-\sigma+\sigma_2}_q} \lesssim (1+\tau)^{p-\frac{np}{2(\sigma-\sigma_1)}(\frac{1}{m}-\frac{1}{qp})}\|u\|_{X(t)}^p.\label{Main.Es1.0.3}
 \end{align}
 From estimates (\ref{Main.Es1.0.1}), (\ref{Main.Es1.0.2}), (\ref{Main.Es1.0.3}) and result (\ref{result1.0.2}) we obtain
 \begin{align*}
     \| u^{\rm non}(t,\cdot)\|_{L^q} &\lesssim (1+t)^{1-\frac{n}{2(\sigma-\sigma_1)}(\frac{1}{m}-\frac{1}{q})} \|u\|_{X(t)}^p,\\
      \||D|^s u^{\rm non}(t,\cdot)\|_{L^q} &\lesssim (1+t)^{1-\frac{n}{2(\sigma-\sigma_1)}(\frac{1}{m}-\frac{1}{q})-\frac{s}{2(\sigma-\sigma_1)}} \|u\|_{X(t)}^p.
 \end{align*}
 From here we get the estimate (\ref{estimates1.1.6}). Next, we need to prove the estimate (\ref{estimates3.1.2}). Using again the $(L^m-L^q)-L^q$ estimates if $\tau \in [0, [t-1]^{+}]$ and the $L^q - L^q$ estimates if $\tau \in [[t-1]^{+}, t]$ from Corollary \ref{corollary2.2}, we derive for two functions $u$ and $v$ from $X(t)$ the following estimate:
\begin{align*}
    &\left\||D|^{ks}(Nu(t,\cdot)-Nv(t,\cdot))\right\|_{L^q} \\
    &\quad \lesssim \int_0^{[t-1]^{+}} (1+t-\tau)^{1-\frac{n}{2(\sigma-\sigma_1)}(\frac{1}{m}-\frac{1}{q})-\frac{ks}{2(\sigma-\sigma_1)}} \left\||u(\tau, \cdot)|^p-|v(\tau,\cdot)|^p\right\|_{L^m \cap L^q} d\tau\\
    &\qquad + \int_{[t-1]^{+}}^t \left\||u(\tau, \cdot)|^p-|v(\tau,\cdot)|^p\right\|_{L^q \cap \dot{H}^{[s-\sigma+\sigma_2]^{+}}_q} d\tau,
\end{align*}
where $k = 0,1$. By using H\"{o}lder's inequality, we get
\begin{align*}
\big\||u(\tau,\cdot)|^p-|v(\tau,\cdot)|^p\big\|_{L^q}& \lesssim \|u(\tau,\cdot)-v(\tau,\cdot)\|_{L^{qp}} \big(\|u(\tau,\cdot)\|^{p-1}_{L^{qp}}+\|v(\tau,\cdot)\|^{p-1}_{L^{qp}}\big),\\
\big\||u(\tau,\cdot)|^p-|v(\tau,\cdot)|^p\big\|_{L^m}& \lesssim \|u(\tau,\cdot)-v(\tau,\cdot)\|_{L^{mp}} \big(\|u(\tau,\cdot)\|^{p-1}_{L^{mp}}+\|v(\tau,\cdot)\|^{p-1}_{L^{mp}}\big).
\end{align*}
Applying the fractional Gagliardo-Nirenberg inequality to the terms
$$ \|u(\tau,\cdot)-v(\tau,\cdot)\|_{L^\eta }, \text{ }\|u(\tau,\cdot)\|_{L^\eta}, \text{ }\|v(\tau,\cdot)\|_{L^\eta} $$
with $\eta=qp$ or $\eta=mp$ we can conclude the following estimates:
\begin{align}
    \big\||u(\tau,\cdot)|^p-|v(\tau,\cdot)|^p\big\|_{L^m} \lesssim& (1+\tau)^{p-\frac{n}{2m(\sigma-\sigma_1)}(p-1)}\|u-v\|_{X(t)}\left(\|u\|_{X(t)}^{p-1}+\|v\|_{X(t)}^{p-1}\right),\label{Main.Es1.0.4}\\ 
     \big\||u(\tau,\cdot)|^p-|v(\tau,\cdot)|^p\big\|_{L^q} \lesssim& (1+\tau)^{p-\frac{np}{2(\sigma-\sigma_1)}(\frac{1}{m}-\frac{1}{qp})}\|u-v\|_{X(t)}\left(\|u\|_{X(t)}^{p-1}+\|v\|_{X(t)}^{p-1}\right)\label{Main.Es.1.0.5}.
\end{align}
We now focus on our attention to estimate $\big\||u(\tau,\cdot)|^p-|v(\tau,\cdot)|^p\big\|_{\dot{H}^{s-\sigma+\sigma_2}_q}$ with $s > \sigma-\sigma_2$. By using the integral representation
$$ |u(\tau,x)|^p-|v(\tau,x)|^p=p\int_0^1 \big(u(\tau,x)-v(\tau,x)\big)G\big(\omega u(\tau,x)+(1-\omega)v(\tau,x)\big)d\omega, $$
where $G(u)=u|u|^{p-2}$, we obtain
$$\big\||u(\tau,\cdot)|^p-|v(\tau,\cdot)|^p\big\|_{\dot{H}^{s-\sigma+\sigma_2}_q} \lesssim \int_0^1 \Big\||D|^{s-\sigma+\sigma_2}\Big(\big(u(\tau,\cdot)-v(\tau,\cdot)\big)G\big(\omega u(\tau,\cdot)+(1-\omega)v(\tau,\cdot)\big)\Big)\Big\|_{L^q}d\omega. $$
\begin{itemize}[leftmargin=*]
    \item If $s-\sigma+\sigma_2 \leq \frac{n}{q}$, thanks to the fractional Leibniz rule with $p > 1+\lceil s-\sigma+\sigma_2 \rceil$ we can proceed as follow:
\begin{align*}
    &\left\| | u(\tau,\cdot)|^p -|v(\tau,\cdot)|^p\right\|_{\dot{H}^{s-\sigma+\sigma_2}_q}\\
    &\qquad\lesssim \| |D|^{s-\sigma+\sigma_2}\left(u(\tau,\cdot)-v(\tau,\cdot)\right)\|_{L^{r_1}} \int_0^1 \left\|G\left(\omega u(\tau,\cdot)+(1-\omega)
v(\tau,\cdot)\right)\right\|_{L^{r_2}} d\omega\\
&\qquad\qquad+ \| u(\tau,\cdot)- v(\tau,\cdot)\|_{L^{r_3}} \int_0^1 \left\| |D|^{s-\sigma+\sigma_2}G\left(\omega u(\tau,\cdot)+(1-\omega) v(\tau,\cdot)\right)\right\|_{L^{r_4}} d\omega\\
&\qquad\lesssim \| |D|^{s-\sigma+\sigma_2}\left(u(\tau,\cdot)-v(\tau,\cdot)\right)\|_{L^{r_1}}\left(\|u(\tau,\cdot)\|_{L^{r_2(p-1)}}^{p-1}+ \|v(\tau,\cdot)\|_{L^{r_2(p-1)}}^{p-1}\right)\\
&\qquad\quad + \| |u(\tau,\cdot)- v(\tau,\cdot)\|_{L^{r_3}} \int_0^1 \left\| |D|^{s-\sigma+\sigma_2}G\left(\omega u(\tau,\cdot)+(1-\omega) v(\tau,\cdot)\right)\right\|_{L^{r_4}} d\omega,
\end{align*}
where 
\begin{equation*}
    \frac{1}{r_1}+\frac{1}{r_2}= \frac{1}{r_3} +\frac{1}{r_4} = \frac{1}{q}.
\end{equation*}
Employing the fractional Garliardo-Nirenberg inequality implies
\begin{align*}
   \| |D|^{s-\sigma+\sigma_2}\left(u(\tau,\cdot)-v(\tau,\cdot)\right)\|_{L^{r_1}} &\lesssim \|u(\tau,\cdot)-v(\tau,\cdot)\|_{\dot{H}^s_q}^{a_1} \|u(\tau,\cdot)-v(\tau,\cdot)\|_{L^q}^{1-a_1},\\
   \|u(\tau,\cdot)\|_{L^{r_2(p-1)}} &\lesssim \|u(\tau,\cdot)\|_{\dot{H}^{s}_q}^{a_2} \|u(\tau,\cdot)\|_{L^q}^{1-a_2},\\
   \left\| u(\tau,\cdot)- v(\tau,\cdot)\right\|_{L^{r_3}} &\lesssim \|u(\tau,\cdot)-v(\tau,\cdot)\|_{\dot{H}^s_q}^{a_3} \|u(\tau,\cdot)-v(\tau,\cdot)\|_{L^q}^{1-a_3},
\end{align*}
where 
\begin{align*}
    a_1 &= \frac{n}{s}\left(\frac{1}{q}-\frac{1}{r_1}+\frac{s-\sigma+\sigma_2}{n}\right) \in \left[\frac{s-\sigma+\sigma_2}{s}, 1\right], \\
    a_2 &= \frac{n}{s}\left(\frac{1}{q}-\frac{1}{r_2(p-1)}\right) \in \left[0, 1\right],\\
    a_3 &= \frac{n}{s}\left(\frac{1}{q}-\frac{1}{r_3}\right) \in \left[0, 1\right].
\end{align*}
Moreover, since $\omega \in [0,1]$ is a parameter, we may apply again the fractional chain rule with $p > 1+\lceil s-\sigma+\sigma_2 \rceil$ and the fractional Gagliardo-Nirenberg inequality to conclude
\begin{align*}
    &\left\| |D|^{s-\sigma+\sigma_2}G\left(\omega u(\tau,\cdot)+(1-\omega) v(\tau,\cdot)\right)\right\|_{L^{r_4}} \\
    &\qquad\lesssim \left\| \omega u(\tau,\cdot)+(1-\omega) v(\tau,\cdot)\right\|_{L^{r_5}}^{p-2} \left\||D|^{s-\sigma+\sigma_2}\left(\omega u(\tau,\cdot)+(1-\omega) v(\tau,\cdot)\right) \right\|_{L^{r_6}}\\
    &\qquad\lesssim \| \omega u(\tau,\cdot)+(1-\omega) v(\tau,\cdot)\|_{\dot{H}^s_q}^{(p-2)a_5+a_6} \|\omega u(\tau,\cdot) +(1-\omega)v(\tau,\cdot)\|_{L^q}^{(p-2)(1-a_5)+1-a_6},
\end{align*}
where 
\begin{align*}
    a_5 &= \frac{n}{s}\left(\frac{1}{q}-\frac{1}{r_5}\right) \in \left[0, 1\right],\\
    a_6 &= \frac{n}{s}\left(\frac{1}{q}-\frac{1}{r_6}+\frac{s-\sigma+\sigma_2}{n}\right) \in \left[\frac{s-\sigma+\sigma_2}{s}, 1\right],
\end{align*}
and 
\begin{equation*}
      \frac{p-2}{r_5}+\frac{1}{r_6} = \frac{1}{r_4}.
\end{equation*}
Hence, we derive 
\begin{align*}
    &\int_0^1 \left\| |D|^{s-\sigma+\sigma_2}G\left(\omega u(\tau,\cdot)+(1-\omega) v(\tau,\cdot)\right)\right\|_{L^{r_4}} d\omega\\
    &\qquad\lesssim \left(\|u(\tau,\cdot)\|_{\dot{H}^s_q}+\|v(\tau,\cdot)\|_{\dot{H}^s_q}\right)^{(p-2)a_5+a_6} \left(\|u(\tau,\cdot)\|_{L^q}+\|v(\tau,\cdot)\|_{L^q}\right)^{(p-2)a_5+1-a_6}.
\end{align*}
Therefore, we conclude
\begin{align*}
    &\| |u(\tau,\cdot)|^p-|v(\tau,\cdot)|^p \|_{\dot{H}^{s-\sigma+\sigma_2}_q} \notag\\
    &\qquad\lesssim (1+\tau)^{p-\frac{np}{2(\sigma-\sigma_1)}(\frac{1}{m}-\frac{1}{qp})-\frac{s-\sigma+\sigma_2}{2(\sigma-\sigma_1)}} \|u-v\|_{X(t)}\left(\|u\|_{X(t)}^{p-1} + \|v\|_{X(t)}^{p-1}\right).
\end{align*}
\item If $s-\sigma+\sigma_2 > \frac{n}{q}$, thanks to the fractional Leibniz formula, we can proceed as follows:
\begin{align*}
&\big\||u(\tau,\cdot)|^p-|v(\tau,\cdot)|^p\big\|_{\dot{H}^{s-\sigma+\sigma_2}_q} \\
&\qquad\lesssim \big\||D|^{s-\sigma+\sigma_2}\big(u(\tau,\cdot)-v(\tau,\cdot)\big)\big\|_{L^{q}} \int_0^1 \big\|G\big(\omega u(\tau,\cdot)+(1-\omega)v(\tau,\cdot)\big)\big\|_{L^{\infty}}d\omega\\
&\qquad\quad + \|u(\tau,\cdot)-v(\tau,\cdot)\|_{L^{q}}  \int_0^1 \big\||D|^{s-\sigma+\sigma_2}G\big(\omega u(\tau,\cdot)+(1-\omega)v(\tau,\cdot)\big)\big\|_{L^{\infty}}d\omega\\
&\qquad\lesssim \big\||D|^{s-\sigma+\sigma_2}\big(u(\tau,\cdot)-v(\tau,\cdot)\big)\big\|_{L^{q}} \Big(\|u(\tau,\cdot)\|^{p-1}_{L^{\infty}}+ \|v(\tau,\cdot)\|^{p-1}_{L^{\infty}}\Big)\\
&\qquad\quad + \|u(\tau,\cdot)-v(\tau,\cdot)\|_{L^{\infty}}  \int_0^1 \big\||D|^{s-\sigma+\sigma_2}G\big(\omega u(\tau,\cdot)+(1-\omega)v(\tau,\cdot)\big)\big\|_{L^{q}}d\omega.
\end{align*}
Employing the fractional Sobolev embedding and the fractional Gagliardo-Nirenberg inequality imply
\begin{align*}
\|u(\tau,\cdot)\|_{L^{\infty}}&\lesssim \|u(\tau,\cdot)\|_{\dot{H}^{s^{*}}_q} + \|u(\tau,\cdot)\|_{\dot{H}^{s-\sigma+\sigma_2}_q}\\
&\lesssim \|u(\tau,\cdot)\|_{L^q}^{1-\theta_2}\||D|^su(\tau,\cdot)\|_{L^q}^{\theta_2} \\
&\lesssim (1+\tau)^{1-\frac{n}{2(\sigma-\sigma_1)}(\frac{1}{m}-\frac{1}{q})-\frac{s^{*}}{2(\sigma-\sigma_1)}}\|u\|_{X(t)}, \\
\|u(\tau,\cdot)-v(\tau,\cdot)\|_{L^{\infty}} &\lesssim (1+\tau)^{1-\frac{n}{2(\sigma-\sigma_1)}(\frac{1}{m}-\frac{1}{q})-\frac{s^{*}}{2(\sigma-\sigma_1)}}\|u-v\|_{X(t)},
\end{align*}
where $\theta_2 = \frac{s^{*}}{s}$. Moreover, since $\omega \in [0,1]$ is a parameter, we may apply again the fractional chain rule  with $p >1+ s-\sigma+\sigma_2$ and the fractional Gagliardo-Nirenberg inequality to conclude
\begin{align*}
&\big\||D|^{s-\sigma+\sigma_2}G\big(\omega u(\tau,\cdot)+(1-\omega)v(\tau,\cdot)\big)\big\|_{L^{q}}\\
&\qquad \lesssim \|\omega u(\tau,\cdot)+(1-\omega)v(\tau,\cdot)\|^{p-2}_{L^{\infty}}\,\, \big\||D|^{s}\big(\omega u(\tau,\cdot)+(1-\omega)v(\tau,\cdot)\big)\big\|_{L^{q}}\\
&\qquad \lesssim (1+\tau)^{(p-2)(1-\frac{n}{2(\sigma-\sigma_1)}(\frac{1}{m}-\frac{1}{q})-\frac{s^{*}}{2(\sigma-\sigma_1)})+1-\frac{n}{2(\sigma-\sigma_1)}(\frac{1}{m}-\frac{1}{q})-\frac{s-\sigma+\sigma_2}{2(\sigma-\sigma_1)})}\left(\|u\|_{X(t)}^{p-1}+\|v\|_{X(t)}^{p-1}\right).
\end{align*}
By choosing $s^{*}= \frac{n}{q}-\epsilon$ with a sufficiently small positive number $\epsilon$, we conclude
\begin{equation*}
\big\||u(\tau,\cdot)|^p-|v(\tau,\cdot)|^p\big\|_{\dot{H}^{s-\sigma+\sigma_2}_q} \lesssim (1+\tau)^{p-\frac{np}{2(\sigma-\sigma_1)}(\frac{1}{m}-\frac{1}{qp})}\|u-v\|_{X(t)}\big( \|u\|^{p-1}_{X(t)}+ \|v\|^{p-1}_{X(t)} \big).
\end{equation*}
\end{itemize}
Summarizing, we get estimates
\begin{align}
    \big\||u(\tau,\cdot)|^p-|v(\tau,\cdot)|^p\big\|_{\dot{H}^{s-\sigma+\sigma_2}_q} \lesssim (1+\tau)^{p-\frac{np}{2(\sigma-\sigma_1)}(\frac{1}{m}-\frac{1}{qp})}\|u-v\|_{X(t)}\big( \|u\|^{p-1}_{X(t)}+ \|v\|^{p-1}_{X(t)} \big).\label{Main.es1.0.6}
\end{align}
From estimates (\ref{Main.Es1.0.4}), (\ref{Main.Es.1.0.5}) and (\ref{Main.es1.0.6}) we obtain
\begin{align*}
     \left\||D|^{ks}(Nu(t,\cdot)-Nv(t,\cdot))\right\|_{L^q} \lesssim (1+t)^{1-\frac{n}{2(\sigma-\sigma_1)}(\frac{1}{m}-\frac{1}{q})-\frac{ks}{2(\sigma-\sigma_1)}}\|u-v\|_{X(t)} \left(\|u\|_{X(t)}^{p-1} +\|v\|_{X(t)}^{p-1}\right),
\end{align*}
where $k = 0,1$. From here we get the estimate (\ref{estimates3.1.2}). Summarizing, Theorem \ref{theorem1.0} is proved completely.

\subsection{Proof of Theorem \ref{theorem1.2}} We introduce the data space $\frak{D}_{m,q}^{2\sigma_2}$ and the solution space
\begin{equation*}
    X(t) := C([0,t], H^s_q) \cap C^1([0,t], L^q),
\end{equation*}
where the weight $f_{4,s}(\tau) = 0$. We have
\begin{align*}
    &\left\|\partial_t^j |D|^{2k\sigma_2} u^{\rm non}(\tau,\cdot) \right\|_{L^q}\\
    &\qquad\lesssim \int_0^{[t-1]^{+}} (1+t-\tau)^{1-\frac{n}{2(\sigma-\sigma_1)}(\frac{1}{m}-\frac{1}{q})-\frac{2k\sigma_2+j\sigma_1}{2(\sigma-\sigma_1)}} \left\||u(\tau, \cdot)|^p\right\|_{L^m \cap L^q} d\tau\\
    &\qquad\quad+ \int_{[t-1]^{+}}^t \left\||u(\tau, \cdot)|^p\right\|_{L^q \cap \dot{H}^{3\sigma_2-\sigma}_q} d\tau,
\end{align*}
where $(i,j) \in \{(0,1),(1,0), (0,0)\}$.
Like the proof steps of Theorem \ref{theorem1.0}, applying the Gagliardo-Nirenberg inequality with condition (\ref{condition1.2.2}) being satisfied, we have the conclusion
\begin{align}
    \|u(\tau,\cdot)\|_{L^{mp}}^p \lesssim& (1+\tau)^{p -\frac{n}{2m(\sigma-\sigma_1)}(p-1)} \|u\|_{X(t)}^p,\label{Main.Es1.3.1}\\
    \|u(\tau,\cdot)\|_{L^{qp}}^p \lesssim& (1+\tau)^{p-\frac{np}{2(\sigma-\sigma_1)}(\frac{1}{m}-\frac{1}{qp})} \|u\|_{X(t)}^p \label{Main.Es1.3.2}.
\end{align}
Next we estimate the norm $\||u(\tau,\cdot)|^p\|_{\dot{H}^{3\sigma_2-\sigma}_q}$ with $s > \sigma-\sigma_2$. 
\begin{itemize}[leftmargin=*]
    \item If $3\sigma_2-\sigma > \frac{n}{q}$, applying the fractional powers rule with $3\sigma_2-\sigma \in \left(\frac{n}{q}, p\right)$ and the fractional Sobolev embedding with a suitable $s^{*} < \frac{n}{q}$ we obtain
\begin{equation*}
    \||u(\tau,\cdot)|^p\|_{\dot{H}^{3\sigma_2-\sigma}_q} \lesssim \|u(\tau,\cdot)\|_{\dot{H}^{3\sigma_2-\sigma}_q}\left(\|u(\tau,\cdot)\|_{\dot{H}^{s^{*},q}}+\|u(\tau,\cdot)\|_{\dot{H}^{3\sigma_2-\sigma}_q}\right)^{p-1}.
\end{equation*}
Applying the fractional Gagliardo-Nirenberg inequality we have
\begin{align*}
\|u(\tau,\cdot)\|_{\dot{H}^{3\sigma_2-\sigma}_q} &\lesssim \|u(\tau,\cdot)\|_{L^q}^{1-\theta_1}\|u(\tau,\cdot)\|_{\dot{H}^s_q}^{\theta}\\
&\lesssim (1+\tau)^{1-\frac{n}{2(\sigma-\sigma_1)}(\frac{1}{m}-\frac{1}{q})-\frac{3\sigma_2-\sigma}{2(\sigma-\sigma_1)}}\|u\|_{X(t)},\\
    \|u(\tau,\cdot)\|_{\dot{H}^{s^{*}, q}} &\lesssim \|u(\tau,\cdot)\|_{L^q}^{1-\theta_2}\||D|^su(\tau,\cdot)\|_{L^q}^{\theta_2} \lesssim (1+\tau)^{1-\frac{n}{2(\sigma-\sigma_1)}(\frac{1}{m}-\frac{1}{q})-\frac{s^{*}}{2(\sigma-\sigma_1)}}\|u\|_{X(t)},
\end{align*}
where $\theta_2 = \frac{s^{*}}{s}$. Hence, we derive
\begin{align}
    \||u(\tau,\cdot)|^p\|_{\dot{H}^{3\sigma_2-\sigma}_q} &\lesssim (1+\tau)^{p(1-\frac{n}{2(\sigma-\sigma_1)}(\frac{1}{m}-\frac{1}{q}))-\frac{3\sigma_2-\sigma}{2(\sigma-\sigma_1)}-(p-1)\frac{s^{*}}{2(\sigma-\sigma_1)}}\|u\|_{X(t)}^p\notag\\
    &\lesssim (1+\tau)^{p-\frac{np}{2(\sigma-\sigma_1)}(\frac{1}{m}-\frac{1}{qp})-\frac{3\sigma_2-\sigma}{2(\sigma-\sigma_1)}}\|u\|_{X(t)}^p, \label{Main.Es1.3.3}
\end{align}
where we choose $s^{*}= \frac{n}{q}-\epsilon$ with a sufficiently small positive number $\epsilon$.
\item If $3\sigma_2-\sigma \leq \frac{n}{q}$, carrying out the same proof steps for $s=2\sigma_2$ as in Theorem \ref{theorem1.0}, we get estimate (\ref{Main.Es1.3.3}).
\end{itemize}
From estimates (\ref{Main.Es1.3.1})-(\ref{Main.Es1.3.3}), we get estimate (\ref{estimates1.1.6}). Now we prove the estimate (\ref{estimates3.1.2}). Using again the $(L^m-L^q)-L^q$ estimates if $\tau \in [0, [t-1]^{+}]$ and the $L^q - L^q$ estimates if $\tau \in [[t-1]^{+}, t]$ from corollary \ref{corollary2.2}, we derive for two functions $u$ and $v$ from $X(t)$ the following estimate:
\begin{align*}
    &\left\|\partial_t^j |D|^{2k\sigma_2} (Nu(\tau,\cdot)-Nv(\tau,\cdot)) \right\|_{L^q}\\
    &\qquad\lesssim \int_0^{[t-1]^{+}} (1+t-\tau)^{1-\frac{n}{2(\sigma-\sigma_1)}(\frac{1}{m}-\frac{1}{q})-\frac{2k\sigma_2+j\sigma_1}{2(\sigma-\sigma_1)}} \left\||u(\tau, \cdot)|^p-|v(\tau,\cdot)|^p\right\|_{L^m \cap L^q} d\tau\\
    &\qquad\quad+ \int_{[t-1]^{+}}^t \left\||u(\tau, \cdot)|^p-|v(\tau,\cdot)|^p\right\|_{L^q \cap \dot{H}^{3\sigma_2-\sigma}_q} d\tau,
\end{align*}
where $(j, k) \in \{(0,0), (0,1), (1,0)\}$. From the proof of Theorem \ref{theorem1.0} we obtain
\begin{align}
    \big\||u(\tau,\cdot)|^p-|v(\tau,\cdot)|^p\big\|_{L^m} \lesssim& (1+\tau)^{p-\frac{n}{2m(\sigma-\sigma_1)}(p-1)}\|u-v\|_{X(t)}\left(\|u\|_{X(t)}^{p-1}+\|v\|_{X(t)}^{p-1}\right),\label{Main.Es1.3.4}\\ 
     \big\||u(\tau,\cdot)|^p-|v(\tau,\cdot)|^p\big\|_{L^q} \lesssim& (1+\tau)^{p-\frac{np}{2(\sigma-\sigma_1)}(\frac{1}{m}-\frac{1}{qp})}\|u-v\|_{X(t)}\left(\|u\|_{X(t)}^{p-1}+\|v\|_{X(t)}^{p-1}\right).\label{Main.Es1.3.5}
     \end{align}
    Finally we need to estimate the norm $\| |u(\tau,\cdot)|^p-|v(\tau,\cdot)|^p\|_{\dot{H}^{3\sigma_2-\sigma}_q}$. Carrying out the same proof steps of theorem \ref{theorem1.0} we have the conclusion
    \begin{align}
        &\| |u(\tau,\cdot)|^p-|v(\tau,\cdot)|^p \|_{\dot{H}^{3\sigma_2-\sigma}_q} \notag\\
    &\qquad\lesssim (1+\tau)^{p-\frac{np}{2(\sigma-\sigma_1)}(\frac{1}{m}-\frac{1}{qp})-\frac{3\sigma_2-\sigma}{2(\sigma-\sigma_1)}} \|u-v\|_{X(t)}\left(\|u\|_{X(t)}^{p-1} + \|v\|_{X(t)}^{p-1}\right). \label{Main.Es1.3.6}
    \end{align}
    From estimates (\ref{Main.Es1.3.4})-(\ref{Main.Es1.3.6}) we get estimate (\ref{estimates3.1.2}). Summarizing, Theorem \ref{theorem1.2} is proved completely.
\subsection{Proof of Theorem \ref{theorem1.3}} We introduction the data space $\frak{D}^{s}_{m,q}$ and the solution space 
\begin{align*}
    X(t) := C([0,t], H^s_q) \cap C^1([0,t], H^{s-2\sigma_2}_q).
\end{align*}
 Let us prove the inequality (\ref{estimates1.1.6}). Applying Corollary \ref{corollary2.2} we have
 \begin{align*}
      &\left\|\partial_t^j |D|^{ks} u^{\rm non}(\tau,\cdot) \right\|_{L^q}\\
    &\qquad\lesssim \int_0^{[t-1]^{+}} (1+t-\tau)^{1-\frac{n}{2(\sigma-\sigma_1)}(\frac{1}{m}-\frac{1}{q})-\frac{ks+j\sigma_1}{2(\sigma-\sigma_1)}} \left\||u(\tau, \cdot)|^p\right\|_{L^m \cap L^q \cap \dot{H}^{s-2\sigma_2}_q} d\tau\\
    &\qquad\quad+ \int_{[t-1]^{+}}^t \left\||u(\tau, \cdot)|^p\right\|_{L^q \cap \dot{H}^{3\sigma_2-\sigma}_q\cap \dot{H}^{s-\sigma+\sigma_2}_q} d\tau
 \end{align*}
 and 
 \begin{align*}
     \||D|^{s-2\sigma_2} u^{\rm non}(t,\cdot)\|_{L^q} \lesssim& \int_0^{[t-1]^{+}} (1+t-\tau)^{1-\frac{n}{2(\sigma-\sigma_1)}(\frac{1}{m}-\frac{1}{q})-\frac{s}{2(\sigma-\sigma_1)}+\frac{\sigma_2-\sigma_1}{\sigma-\sigma_1}} \| |u(\tau,\cdot)|^p\|_{L^m \cap L^q \cap \dot{H}^{s-2\sigma_2}_q} d\tau\\
     &+ \int_{[t-1]^{+}}^t \| |u(\tau,\cdot)|^p\|_{L^q \cap \dot{H}^{s-\sigma+\sigma_2}_q} d\tau,
 \end{align*}
 where $(k,j) \in \{(0,0), (1,0), (0,1)\}$
 From the proof of Theorem \ref{theorem1.0} we obtain
\begin{align}
    \|u(\tau,\cdot)\|_{L^{mp}}^p
    \lesssim &(1+\tau)^{p -\frac{n}{2m(\sigma-\sigma_1)}(p-1)} \|u\|_{X(t)}^p, \label{Main.Es1.4.1}\\
    \|u(\tau,\cdot)\|_{L^{qp}}^p
    \lesssim & (1+\tau)^{p-\frac{np}{2(\sigma-\sigma_1)}(\frac{1}{m}-\frac{1}{qp})} \|u\|_{X(t)}^p, \label{Main.Es1.4.2}
\end{align}
with the following condition being satisfied:
\begin{align*}
    p \in \left[\frac{q}{m}, \infty\right) \text{ if } n = qs \text{ or } p \in \left[\frac{q}{m}, \frac{n}{n-qs}\right] \text{ if } n \in \left(qs, \frac{q^2s}{q-m}\right].
\end{align*}
Now we estimate the norms $\| |u(\tau,\cdot)|^p\|_{\dot{H}^{s-\sigma+\sigma_2}_q}$, $\| |u(\tau,\cdot)|^p\|_{\dot{H}^{s-2\sigma_2}_q}$ and  $\| |u(\tau,\cdot)|^p\|_{\dot{H}^{3\sigma_2-\sigma}_q}$. Carrying out the same proof steps as in Theorem \ref{theorem1.0} we obtain
\begin{align*}
    \||u(\tau,\cdot)|^p\|_{\dot{H}^{s-\sigma+\sigma_2}_q}
     \lesssim& (1+\tau)^{p-\frac{np}{2(\sigma-\sigma_1)}(\frac{1}{m}-\frac{1}{qp})-\frac{s-\sigma+\sigma_2}{2(\sigma-\sigma_1)}} \|u\|_{X(t)}^p.
\end{align*}
From here we have a conclusion
\begin{align}
    \max\left\{ \| |u(\tau,\cdot)|^p\|_{\dot{H}^{s-2\sigma_2}_q}, \| |u(\tau,\cdot)|^p\|_{\dot{H}^{3\sigma_2-\sigma}_q} \right\} &\lesssim \| |u(\tau,\cdot)|^p\|_{H^{s-\sigma+\sigma_2}_q}\notag\\
    &\lesssim (1+\tau)^{p-\frac{np}{2(\sigma-\sigma_1)}(\frac{1}{m}-\frac{1}{qp})} \|u\|_{X(t)}^p. \label{Main.Es1.4.3}
\end{align}
The condition (\ref{condition1.3.1}) implies
\begin{align*}
    p-\frac{n}{2m(\sigma-\sigma_1)}(p-1) &< -1,\\
    p-\frac{np}{2(\sigma-\sigma_1)}\left(\frac{1}{m}-\frac{1}{qp}\right) &< 1-\frac{n}{2(\sigma-\sigma_1)}\left(\frac{1}{m}-\frac{1}{q}\right)-\frac{s}{2(\sigma-\sigma_1)},
\end{align*}
combined with estimates (\ref{Main.Es1.4.1})-(\ref{Main.Es1.4.3}) we have estimate (\ref{estimates1.1.6}). Carrying out the same proof steps as theorem \ref{theorem1.0}, we also get estimate (\ref{estimates3.1.2}). Summarizing, Theorem \ref{theorem1.3} is proved completely.
%.................................................................................
\section{Some generalizations}\label{Sec.4}
In this section, to generalize some results for \eqref{Main.Eq.1} let us consider the following Cauchy problem for semi-linear $\sigma$-evolution equations with double damping:
\begin{equation}\label{Main.Eq.2}
    \begin{cases}
        u_{tt}+ (-\Delta)^\sigma u+ \mu_1(-\Delta)^{\sigma_1} u_t+ \mu_2(-\Delta)^{\sigma_2} u_t= ||D|^a u|^p, &\quad x\in \R^n,\, t \ge 0, \\
u(0,x)= u_0(x),\quad u_t(0,x)= u_1(x), &\quad x\in \R^n, \\
    \end{cases}
\end{equation}
where $\sigma\ge 1$, $\sigma_1$, $\sigma_2$ are any constants satisfying $0< \sigma_1< \sigma/2< \sigma_2< \sigma$. Here $a, \mu_1,\mu_2 $ are positive constants and the parameter $p>1$ stands for power exponents of the nonlinear term.
\begin{theorem}\label{theorem1.5}
    Let $0 \leq a < \sigma-\sigma_2$. Let $q \in (1, \infty)$ be a fixed constant and $m \in [1, q)$. We assume that the exponent $p$ satisfies the conditions $p > 1 + \lceil 3\sigma_2-\sigma \rceil$ and
    \begin{equation}\label{condition1.5.1}
        p > 1+ \frac{4m(\sigma-\sigma_1)}{n-2m(\sigma-\sigma_1)+ma} \text{ and } n > 2m(\sigma-\sigma_1) - ma.
    \end{equation}
    Moreover, we suppose the following conditions:
    \begin{align}
        &p \in \left[\frac{q}{m}, \infty\right)&\quad &\text{ if } n \leq q(2\sigma_2-a),& \notag\\
        &p \in \left[\frac{q}{m}, \frac{n+q(\sigma-3\sigma_2)}{n-q(2\sigma_2-a)}\right] &\quad &\text{ if } n \in \left(q(2\sigma_2-a), \frac{q^2(2\sigma_2-a)-qm(3\sigma_2-\sigma)}{q-m}\right].&\label{condition1.5.2}
    \end{align}
    Then, there extists a constant $\epsilon > 0$ such that for any small data
    \begin{equation*}
        (u_0, u_1) \in \frak{D}_{m,q}^{2\sigma_2}  \text{ satisfying the asumption } \|(u_0,u_1)\|_{\frak{D}_{m,q}^{2\sigma_2}} \leq \epsilon,
    \end{equation*}
    we have a uniquely determined global (in time) small data energy solution
    \begin{equation*}
        u \in C([0, \infty), H^{2\sigma_2}_q) \cap C^1([0, \infty), L^q)
    \end{equation*}
    to (\ref{Main.Eq.2}).The following estimates hold:
\begin{align*}
    \|u(t,\cdot)\|_{L^q} &\lesssim (1+t)^{1-\frac{n}{2(\sigma-\sigma_1)}\left(\frac{1}{m}-\frac{1}{q}\right)} \|(u_0,u_1)\|_{\frak{D}_{m,q}^{2\sigma_2}},\\
    \||D|^{2\sigma_2}  u(t,\cdot)\|_{L^q} &\lesssim (1+t)^{1-\frac{n}{2(\sigma-\sigma_1)}\left(\frac{1}{m}-\frac{1}{q}\right)-\frac{\sigma_2}{\sigma-\sigma_1}} \|(u_0,u_1)\|_{\frak{D}_{m,q}^{2\sigma_2}},\label{estimates1.5.2}\\
    \|u_t(t,\cdot)\|_{L^q} &\lesssim (1+t)^{1-\frac{n}{2(\sigma-\sigma_1)}\left(\frac{1}{m}-\frac{1}{q}\right)-\frac{\sigma_1}{\sigma-\sigma_1}} \|(u_0,u_1)\|_{\frak{D}_{m,q}^{2\sigma_2}}. 
\end{align*}

\end{theorem}
\begin{proof}
    We introduce the data space $\frak{D}_{m,q}^{2\sigma_2}$ and the solution space 
\begin{equation*}
    X(t) := C([0,t], H^{2\sigma_2}_q) \cap C^1([0,t], L^q),
\end{equation*}
where the weight $f_{4,s} = 0$. Let us prove the inequality (\ref{estimates1.1.6}). From Corollary \ref{corollary2.2} we have
\begin{align*}
    \|u^{\rm non}(t,\cdot)\|_{L^q} \lesssim& \int_0^{[t-1]^{+}} (1+t-\tau)^{1-\frac{n}{2(\sigma-\sigma_1)}(\frac{1}{m}-\frac{1}{q})}\left\|||D|^au(\tau,\cdot)|^p\right\|_{L^m \cap L^q} d\tau \\
    &+ \int_{[t-1]^{+}}^t (t-\tau) \left\|||D|^a u(\tau,\cdot)|^p\right\|_{L^q} d\tau.
\end{align*}
Hence, it is necessary to require the estimates for $||D|^au(\tau,\cdot)|^p$ in $L^m \cap L^q$ and $L^q$ as follows:
\begin{align*}
    \|||D|^au(\tau,\cdot)|^p\|_{L^m \cap L^q} &\lesssim \| |D|^a u(\tau,\cdot)\|_{L^{mp}}^p +\||D|^a u(\tau,\cdot)\|_{L^{qp}}^p,\\
\|||D|^au(\tau,\cdot)|^p\|_{L^q} &\lesssim \| |D|^a u(\tau,\cdot)\|_{L^{qp}}^p.
\end{align*}
Applying the fractional Gagliardo- Nirenberg inequality we can conclude 
\begin{align*}
    \| |D|^a u(\tau,\cdot)\|_{L^{mp}}^p \lesssim& \|u\|_{X(t)}^p (1+\tau)^{p-(p-1)\frac{n}{2m(\sigma-\sigma_1)}-\frac{pa}{2(\sigma-\sigma_1)}},\\
    \| |D|^a u(\tau,\cdot)\|_{L^{qp}}^p \lesssim& \|u\|_{X(t)}^p (1+\tau)^{p-\frac{np}{2(\sigma-\sigma_1)}(\frac{1}{m}-\frac{1}{qp})-\frac{pa}{2(\sigma-\sigma_1)}},
\end{align*}
provided that (\ref{condition1.5.2}) is satisfied. From both estimates we may conclude
\begin{align}
    \| u^{\rm non}(t,\cdot)\|_{L^q} \lesssim& \|u\|_{X(t)}^p\int_0^{[t-1]^{+}} (1+t-\tau)^{1-\frac{n}{2(\sigma-\sigma_1)}(\frac{1}{m}-\frac{1}{q})}(1+\tau)^{p-(p-1)\frac{n}{2m(\sigma-\sigma_1)}-\frac{pa}{2(\sigma-\sigma_1)}}d\tau \notag\\
    &+ \|u\|_{X(t)}^p\int_{[t-1]^{+}}^t (t-\tau) (1+\tau)^{p-\frac{np}{2(\sigma-\sigma_1)}(\frac{1}{m}-\frac{1}{qp})-\frac{pa}{2(\sigma-\sigma_1)}}d\tau \notag\\
    \lesssim& \|u\|_{X(t)}^p (1+t)^{1-\frac{n}{2(\sigma-\sigma_1)}(\frac{1}{m}-\frac{1}{q})}.\label{Main.estimate1.5.1}
\end{align}
The second inequality arises as implied by condition (\ref{condition1.5.1})
\begin{align*}
    p-(p-1)\frac{n}{2m(\sigma-\sigma_1)}-\frac{pa}{2(\sigma-\sigma_1)} &< -1,\\
    p-\frac{np}{2(\sigma-\sigma_1)}\left(\frac{1}{m}-\frac{1}{qp}\right)-\frac{pa}{2(\sigma-\sigma_1)} &< -1-\frac{n}{2(\sigma-\sigma_1)}\left(\frac{1}{m}-\frac{1}{q}\right).
\end{align*}
Now, let us turn to estimate the norm $\|u_t^{\rm non}(t,\cdot)\|_{L^q}$. We use the $(L^m \cap L^q)-L^q$ estimate if $\tau \in [0, [t-1]^{+}]$ and the $L^q-L^q$ estimates if $\tau \in \left[[t-1]^{+}, t\right]$ from Corollary \ref{corollary2.2} to derive
\begin{align*}
    \|u_t^{\rm non}(t,\cdot)\|_{L^q} \lesssim& \int_0^{[t-1]^{+}} (1+t-\tau)^{-\frac{n}{2(\sigma-\sigma_1)}(\frac{1}{m}-\frac{1}{q})+1-\frac{\sigma_1}{\sigma-\sigma_1}}\| ||D|^au(\tau,\cdot)|^p\|_{L^m \cap L^q} d\tau\\
    &+ \int_{[t-1]^{+}}^t \|| |D|^a u(\tau,\cdot)|^p\|_{L^q \cap \dot{H}^{3\sigma_2-\sigma}_q} d\tau.
\end{align*}
The first integral is evaluated as follows:
\begin{align*}
    &\int_0^{[t-1]^{+}} (1+t-\tau)^{-\frac{n}{2(\sigma-\sigma_1)}(\frac{1}{m}-\frac{1}{q})+1-\frac{\sigma_1}{\sigma-\sigma_1}}\| ||D|^au(\tau,\cdot)|^p\|_{L^m \cap L^q} d\tau \\
    &\qquad\quad \lesssim \|u\|_{X(t)}^p (1+t)^{-\frac{n}{2(\sigma-\sigma_1)}(\frac{1}{m}-\frac{1}{q})+1-\frac{\sigma_1}{\sigma-\sigma_1}}.
\end{align*}
In the second integral we need to estimate the norm $\| ||D|^a u(\tau,\cdot)|^p\|_{\dot{H}^{3\sigma_2-\sigma}_q}$. Applying the fractional chain rule with $p > \lceil 3\sigma_2-\sigma \rceil$ and the fractional Gagliardo-Nirenberg inequality with $ 0 \leq a < \sigma-\sigma_2 $ we obtain
\begin{align*}
    \left\| ||D|^au(\tau,\cdot)|^p\right\|_{\dot{H}^{3\sigma_2-\sigma}_q} \lesssim& \| |D|^a u(\tau,\cdot)\|_{L^{r_1}}^{p-1} \| |D|^{a+3\sigma_2-\sigma} u(\tau,\cdot)\|_{L^{r_2}} \\
    \lesssim& \|u(\tau,\cdot)\|_{L^q}^{(p-1)(1-\theta_1)} \| |D|^{2\sigma_2} u(\tau,\cdot)\|_{L^q}^{(p-1)\theta_1} \|u(\tau,\cdot)\|_{L^q}^{1-\theta_2}\||D|^{2\sigma_2}u(\tau,\cdot)\|_{L^q}^{\theta_2}\\
    \lesssim& \|u\|_{X(t)}^p (1+\tau)^{p(1-\frac{n}{2(
    \sigma-\sigma_1)}(\frac{1}{m}-\frac{1}{q}))-\sigma_2\frac{(p-1)\theta_1+\theta_2}{\sigma-\sigma_1}}\\
    \lesssim& \|u\|_{X(t)}^p (1+\tau)^{p-\frac{np}{2(\sigma-\sigma_1)}(\frac{1}{m}-\frac{1}{qp})-\frac{pa}{2(\sigma-\sigma_1)}-\frac{3\sigma_2-\sigma}{2(\sigma-\sigma_1)}},
\end{align*}
where 
\begin{equation*}
    \theta_1 = \frac{n}{2\sigma_2}\left(\frac{1}{q}-\frac{1}{r_1}+\frac{a}{n}\right) \in \left[\frac{a}{2\sigma_2}, 1\right] 
    \text{ , } \theta_2 = \frac{n}{2\sigma_2}\left(\frac{1}{q}-\frac{1}{r_2}+\frac{a+3\sigma_2-\sigma}{n}\right) \in \left[\frac{a+3\sigma_2-\sigma}{2\sigma_2}, 1\right]
\end{equation*}
and 
\begin{equation*}
    \frac{p-1}{r_1}+\frac{1}{r_2} =\frac{1}{q}.
\end{equation*}
These conditions imply condition (\ref{condition1.5.2}). As a result, we can conclude 
\begin{equation}\label{Main.estimate1.5.2}
    \|u_t^{\rm non}(t,\cdot)\|_{L^q} \lesssim \|u\|_{X(t)}^p (1+t)^{-\frac{n}{2(\sigma-\sigma_1)}(\frac{1}{m}-\frac{1}{q})+1-\frac{\sigma_1}{\sigma-\sigma_1}}.
\end{equation}
Finally, we need to estimate the norm $\| |D|^{2\sigma_2}u^{\rm non} (t,\cdot)\|_{L^q}$. Carrying out the same calculation steps as above, we have a conclusion
\begin{align}
    \| |D|^{2\sigma_2} u^{\rm non}(t,\cdot)\|_{L^q} \lesssim& \int_0^{[t-1]^{+}} (1+t-\tau)^{-\frac{n}{2(\sigma-\sigma_1)}(\frac{1}{m}-\frac{1}{q})+1-\frac{\sigma_2}{\sigma-\sigma_1}} \left\| | |D|^a u(\tau,\cdot)|^p\right\|_{L^m \cap L^q}
    \notag\\
    &+ \int_{[t-1]^{+}}^t \left\| | |D|^a u(\tau,\cdot)|^p\right\|_{L^q \cap \dot{H}^{3\sigma_2-\sigma}_q} d\tau \notag\\
    \lesssim& \|u\|_{X(t)}^p (1+t)^{-\frac{n}{2(\sigma-\sigma_1)}(\frac{1}{m}-\frac{1}{q})+1-\frac{\sigma_2}{\sigma-\sigma_1}}.\label{Main.estimate1.5.3}
\end{align}
From estimates (\ref{Main.estimate1.5.1})-(\ref{Main.estimate1.5.3}), we have the conclusion of estimate (\ref{estimates1.1.6}).\medskip

Next we prove the estimate (\ref{estimates3.1.2}). Using again the $(L^m-L^q)-L^q$ estimates if $\tau \in [0, [t-1]^{+}]$ and the $L^q-L^q$ estimates if $\tau \in [[t-1]^{+}, t]$ from corollary \ref{corollary2.2}, we derive for two functions $u$ and $v$ from $X(t)$ the following estimate:
\begin{align*}
    &\left\|Nu(t,\cdot)-Nv(t,\cdot)\right\|_{L^q} \\
    &\qquad\quad\lesssim \int_0^{[t-1]^{+}} (1+t-\tau)^{-\frac{n}{2(\sigma-\sigma_1)}(\frac{1}{m}-\frac{1}{q})+1} \left\| | |D|^a u(\tau,\cdot)|^p - | |D|^a v(\tau,\cdot)|^p\right\|_{L^m \cap L^q} d\tau\\
    &\qquad\quad\quad +\int_{[t-1]^{+}}^t \left\| | |D|^a u(\tau,\cdot)|^p - | |D|^a v(\tau,\cdot)|^p\right\|_{L^q} d\tau, \\
    &\left\|\partial_t^j |D|^{2k\sigma_2}\left(Nu(t,\cdot)-Nv(t,\cdot)\right)\right\|_{L^q} \\
    &\qquad\quad\lesssim\int_0^{[t-1]^{+}} (1+t-\tau)^{-\frac{n}{2(\sigma-\sigma_1)}(\frac{1}{m}-\frac{1}{q})+1-\frac{k\sigma_2+j\sigma_1}{\sigma-\sigma_1}} \left\| | |D|^a u(\tau,\cdot)|^p - | |D|^a v(\tau,\cdot)|^p\right\|_{L^m \cap L^q} d\tau\\
    &\qquad\quad\quad + \int_{[t-1]^{+}}^t \left\| | |D|^a u(\tau,\cdot)|^p - | |D|^a v(\tau,\cdot)|^p\right\|_{ L^q \cap \dot{H}^{3\sigma_2-\sigma}_q} d\tau.
\end{align*}
where $(j,k) = (0,1)$ or $(j,k)=(1,0)$. By using Holder's inequality, we get
\begin{align*}
    &\left\| ||D|^au(\tau,\cdot)|^p - ||D|^av(\tau,\cdot)|^p\right\|_{L^q} \\
    &\qquad \lesssim \||D|^a u(\tau,\cdot)-|D|^av(\tau,\cdot)\|_{L^{qp}} \left(\||D|^au(\tau,\cdot)\|_{L^{qp}}^{p-1} +\||D|^av(\tau,\cdot)\|_{L^{qp}}^{p-1}\right),\\
    &\left\| ||D|^au(\tau,\cdot)|^p - ||D|^av(\tau,\cdot)|^p\right\|_{L^m} \\
    &\qquad \lesssim \||D|^a u(\tau,\cdot)-|D|^av(\tau,\cdot)\|_{L^{mp}} \left(\||D|^au(\tau,\cdot)\|_{L^{mp}}^{p-1} + \||D|^av(\tau,\cdot)\|_{L^{mp}}^{p-1}\right).
\end{align*}
Applying the fractional Gagliardo-Nirenberg inequality to the terms
\begin{equation*}
    \| |D|^a u(\tau,\cdot)-|D|^av(\tau,\cdot)\|_{L^h} \text{ , } \| |D|^au(\tau,\cdot)\|_{L^h} \text{ and } \| |D|^a v(\tau,\cdot)\|_{L^h}
\end{equation*}
with $h= qp$ and $h = mp$ we may conclude
\begin{align}
    &\left\| ||D|^au(\tau,\cdot)|^p - ||D|^av(\tau,\cdot)|^p\right\|_{L^q}  \notag\\
    &\qquad\quad\lesssim (1+\tau)^{p-\frac{np}{2(\sigma-\sigma_1)}(\frac{1}{m}-\frac{1}{qp})-\frac{pa}{2(\sigma-\sigma_1)}} \|u(t,\cdot)-v(t,\cdot)\|_{X(t)}\left(\|u(t,\cdot)\|_{X(t)}^{p-1}+\|v(t,\cdot)\|_{X(t)}^{p-1}\right), \label{Main.estiamte1.5.4}\\
&\left\|||D|^au(\tau,\cdot)|^p - ||D|^av(\tau,\cdot)|^p\right\|_{L^m} \notag\\
&\qquad\quad\lesssim (1+\tau)^{p-(p-1)\frac{n}{2m(\sigma-\sigma_1)}-\frac{pa}{2(\sigma-\sigma_1)}} \|u(t,\cdot)-v(t,\cdot)\|_{X(t)}\left(\|u(t,\cdot)\|_{X(t)}^{p-1}+\|v(t,\cdot)\|_{X(t)}^{p-1}\right).\label{Main.estimate1.5.5}
\end{align}
We now focus on our attention to estimate $\| ||D|^au(\tau,\cdot)|^p-||D|^av(\tau,\cdot)|^p \|_{\dot{H}^{3\sigma_2-\sigma}_q}$. By using the integral representation
\begin{align*}
    &||D|^au(\tau, x)|^p-||D|^av(\tau, x)|^p \\
    &\qquad= p \int_0^1 \left(|D|^a u(\tau,x)-|D|^a v(\tau,x)\right) G\left(\omega |D|^a u(\tau,x)+(1-\omega) |D|^a v(\tau,x)\right) d\omega,
\end{align*}
where $G(u)= u|u|^{p-2}$, we obtain
\begin{align*}
    &\left\| ||D|^a u(\tau,\cdot)|^p - ||D|^av(\tau,\cdot)|^p\right\|_{\dot{H}^{3\sigma_2-\sigma}_q} \\
    &\qquad\quad\lesssim \int_0^1 \left\||D|^{3\sigma_2-\sigma}\left(\left(|D|^a u(\tau,x)-|D|^a v(\tau,x)\right) G\left(\omega |D|^a u(\tau,x)+(1-\omega) |D|^a v(\tau,x)\right) \right)\right\|_{L^q}d\omega.
\end{align*}
Thanks to the fractional Leibniz formula we can proceed as follows:
\begin{align*}
    &\left\| ||D|^a u(\tau,\cdot)|^p - ||D|^av(\tau,\cdot)|^p\right\|_{\dot{H}^{3\sigma_2-\sigma}_q}\\
    &\qquad\lesssim \| |D|^{a+3\sigma_2-\sigma}\left(u(\tau,\cdot)-v(\tau,\cdot)\right)\|_{L^{q_1}} \int_0^1 \left\|G\left(\omega |D|^a u(\tau,\cdot)+(1-\omega)|D|^a
v(\tau,\cdot)\right)\right\|_{L^{q_2}} d\omega\\
&\qquad\qquad+ \| |D|^au(\tau,\cdot)-|D|^a v(\tau,\cdot)\|_{L^{q_3}} \int_0^1 \left\| |D|^{3\sigma_2-\sigma}G\left(\omega |D|^au(\tau,\cdot)+(1-\omega)|D|^a v(\tau,\cdot)\right)\right\|_{L^{q_4}} d\omega\\
&\qquad\lesssim \| |D|^{a+3\sigma_2-\sigma}\left(u(\tau,\cdot)-v(\tau,\cdot)\right)\|_{L^{q_1}}\left(\||D|^au(\tau,\cdot)\|_{L^{q_2(p-1)}}^{p-1}+ \||D|^av(\tau,\cdot)\|_{L^{q_2(p-1)}}^{p-1}\right)\\
&\qquad\quad + \| |D|^au(\tau,\cdot)-|D|^a v(\tau,\cdot)\|_{L^{q_3}} \int_0^1 \left\| |D|^{3\sigma_2-\sigma}G\left(\omega |D|^au(\tau,\cdot)+(1-\omega)|D|^a v(\tau,\cdot)\right)\right\|_{L^{q_4}} d\omega,
\end{align*}
where 
\begin{equation*}
    \frac{1}{q_1}+\frac{1}{q_2}= \frac{1}{q_3} +\frac{1}{q_4} = \frac{1}{q}.
\end{equation*}
Employing the fractional Garliardo-Nirenberg inequality implies
\begin{align*}
   \| |D|^{a+3\sigma_2-\sigma}\left(u(\tau,\cdot)-v(\tau,\cdot)\right)\|_{L^{q_1}} &\lesssim \|u(\tau,\cdot)-v(\tau,\cdot)\|_{\dot{H}^{2\sigma_2}_q}^{a_1} \|u(\tau,\cdot)-v(\tau,\cdot)\|_{L^q}^{1-a_1},\\
   \||D|^au(\tau,\cdot)\|_{L^{q_2(p-1)}} &\lesssim \|u(\tau,\cdot)\|_{\dot{H}^{2\sigma_2}_q}^{a_2} \|u(\tau,\cdot)\|_{L^q}^{1-a_2},\\
   \left\| |D|^a\left(u(\tau,\cdot)- v(\tau,\cdot)\right)\right\|_{L^{q_3}} &\lesssim \|u(\tau,\cdot)-v(\tau,\cdot)\|_{\dot{H}^{2\sigma_2}_q}^{a_3} \|u(\tau,\cdot)-v(\tau,\cdot)\|_{L^q}^{a_3},
\end{align*}
where 
\begin{align*}
    a_1 &= \frac{n}{2\sigma_2}\left(\frac{1}{q}-\frac{1}{q_1}+\frac{a+3\sigma_2-\sigma}{n}\right) \in \left[\frac{a+3\sigma_2-\sigma}{2\sigma_2}, 1\right], \\
    a_2 &= \frac{n}{2\sigma_2}\left(\frac{1}{q}-\frac{1}{q_2(p-1)}+\frac{a}{2\sigma_2}\right) \in \left[\frac{a}{2\sigma_2}, 1\right],\\
    a_3 &= \frac{n}{2\sigma_2}\left(\frac{1}{q}-\frac{1}{q_3}+\frac{a}{2\sigma_2}\right) \in \left[\frac{a}{2\sigma_2}, 1\right].
\end{align*}
Moreover, since $\omega \in [0,1]$ is a parameter, we may apply again the fractional chain rule with $p > 1+\lceil 3\sigma_2-\sigma \rceil$ and the fractional Gagliardo-Nirenberg inequality to conclude
\begin{align*}
    &\left\| |D|^{3\sigma_2-\sigma}G\left(\omega |D|^au(\tau,\cdot)+(1-\omega)|D|^a v(\tau,\cdot)\right)\right\|_{L^{q_4}} \\
    &\qquad\lesssim \left\| \omega |D|^a u(\tau,\cdot)+(1-\omega)|D|^a v(\tau,\cdot)\right\|_{L^{q_5}}^{p-2} \left\||D|^{a+3\sigma_2-\sigma}\left(\omega u(\tau,\cdot)+(1-\omega) v(\tau,\cdot)\right) \right\|_{L^{q_6}}\\
    &\qquad\lesssim \| \omega u(\tau,\cdot)+(1-\omega) v(\tau,\cdot)\|_{\dot{H}^{2\sigma_2}_q}^{(p-2)a_5+a_6} \|\omega u(\tau,\cdot) +(1-\omega)v(\tau,\cdot)\|_{L^q}^{(p-2)(1-a_5)+1-a_6},
\end{align*}
where 
\begin{align*}
    a_5 &= \frac{n}{2\sigma_2}\left(\frac{1}{q}-\frac{1}{q_5}+\frac{a}{2\sigma_2}\right) \in \left[\frac{a}{2\sigma_2}, 1\right],\\
    a_6 &= \frac{n}{2\sigma_2}\left(\frac{1}{q}-\frac{1}{q_6}+\frac{a+3\sigma_2-\sigma}{2\sigma_2}\right) \in \left[\frac{a+3\sigma_2-\sigma}{2\sigma_2}, 1\right]
\end{align*}
and 
\begin{equation*}
      \frac{p-2}{q_5}+\frac{1}{q_6} = \frac{1}{q_4}.
\end{equation*}
Hence, we derive 
\begin{align*}
    &\int_0^1 \left\| |D|^{3\sigma_2-\sigma}G\left(\omega |D|^au(\tau,\cdot)+(1-\omega)|D|^a v(\tau,\cdot)\right)\right\|_{L^{q_4}} d\omega\\
    &\qquad\lesssim \left(\|u(\tau,\cdot)\|_{\dot{H}^{2\sigma_2}_q}+\|v(\tau,\cdot)\|_{\dot{H}^{2\sigma_2}_q}\right)^{(p-2)a_5+a_6} \left(\|u(\tau,\cdot)\|_{L^q}+\|v(\tau,\cdot)\|_{L^q}\right)^{(p-2)a_5+1-a_6}.
\end{align*}
Therefore, we conclude
\begin{align}
    &\| ||D|^au(\tau,\cdot)|^p-||D|^av(\tau,\cdot)|^p \|_{\dot{H}^{3\sigma_2-\sigma}_q} \notag\\
    &\qquad\lesssim (1+\tau)^{p-\frac{np}{2(\sigma-\sigma_1)}(\frac{1}{m}-\frac{1}{qp})-\frac{pa}{2(\sigma-\sigma_1)}-\frac{3\sigma_2-\sigma}{2(\sigma-\sigma_1)}} \|u-v\|_{X(t)}\left(\|u\|_{X(t)}^{p-1} + \|v\|_{X(t)}^{p-1}\right). \label{Main.estimate1.5.6}
\end{align}
From estimates (\ref{Main.estiamte1.5.4})-(\ref{Main.estimate1.5.6}), performing the same proof steps as estimate (\ref{estimates1.1.6}), we have the conclusion of estimate (\ref{estimates3.1.2}). Theorem \ref{theorem1.5} is proved completely.
\end{proof}
%.................................................................
%\section{Concluding remarks}
%.................................................................
\appendix
\section{The results are known}
\begin{lemma}[see \cite{IkehataTakeda2019}] \label{L^1.Lemma}
Let $a\ge 0$. Let us assume $h= h(x) \in L^1$ and $\phi=\phi(t,x)$ be a smooth function satisfying
$$ \big\||D|^a \phi(t,\cdot)\big\|_{L^2} \lesssim t^{-\alpha} \quad \text{ and }\quad  \big\||D|^{a+1} \phi(t,\cdot)\big\|_{L^2} \lesssim t^{-\alpha-\beta}, $$
for some positive constants $\alpha,\,\beta>0$. Then, the following estimate holds:
$$ \left\||D|^a \left(\phi(t,x) \ast h(x)- \left(\int_{\R^n}h(y)\,dy\right)\phi(t,x)\right)(t,\cdot) \right\|_{L^2}= o\big(t^{-\alpha}\big) \quad \text{ as }t \to \ity, $$
for all space dimensions $n\ge 1$.
\end{lemma}

\begin{proposition}[Fractional Gagliardo-Nirenberg inequality] \label{FractionalG-N}
Let $1<p,\, p_0,\, p_1<\infty$, $\sigma >0$ and $s\in [0,\sigma)$. Then, it holds for all $u\in L^{p_0} \cap \dot{H}^\sigma_{p_1}$
$$ \|u\|_{\dot{H}^{s}_p}\lesssim \|u\|_{L^{p_0}}^{1-\theta}\,\, \|u\|_{\dot{H}^{\sigma}_{p_1}}^\theta, $$
where $\theta=\theta_{s,\sigma}(p,p_0,p_1)=\frac{\frac{1}{p_0}-\frac{1}{p}+\frac{s}{n}}{\frac{1}{p_0}-\frac{1}{p_1}+\frac{\sigma}{n}}$ and $\frac{s}{\sigma}\leq \theta\leq 1$.
\end{proposition} 

\begin{proposition}[Fractional Leibniz rule] \label{FractionalLeibniz}
Let us assume $s>0$, $1\leq r \leq \infty$ and $1<p_1,\, p_2,\, q_1,\, q_2 \le \infty$ satisfying the relation
\[ \frac{1}{r}=\frac{1}{p_1}+\frac{1}{p_2}=\frac{1}{q_1}+\frac{1}{q_2}.\]
Then, it holds for any $u\in \dot{H}^s_{p_1} \cap L^{q_1}$ and $v\in \dot{H}^s_{q_2} \cap L^{p_2}$
$$\big\||D|^s(u \,v)\big\|_{L^r}\lesssim \big\||D|^s u\big\|_{L^{p_1}}\, \|v\|_{L^{p_2}}+\|u\|_{L^{q_1}}\, \big\||D|^s v\big\|_{L^{q_2}}. $$    
\end{proposition}

\begin{proposition}[Fractional powers] \label{FractionalPowers}
Let $p>1$, $1< r <\infty$ and $u \in H^{s}_r$, where $s \in \big(\frac{n}{r},p\big)$. Let us denote by $F(u)$ one of the functions $|u|^p,\, \pm |u|^{p-1}u$. Then, the following estimates hold:
$$\|F(u)\|_{H^{s}_r}\lesssim \|u\|_{H^{s}_r}\,\, \|u\|_{L^\infty}^{p-1} \quad \text{ and }\quad \| F(u)\|_{\dot{H}^{s}_r}\lesssim \|u\|_{\dot{H}^{s}_r}\,\, \|u\|_{L^\infty}^{p-1}. $$
\end{proposition}

\begin{proposition}[A fractional Sobolev embedding] \label{Embedding}
Let $0< s_1< \frac{n}{2}< s_2$. Then, for any function $u \in \dot{H}^{s_1} \cap \dot{H}^{s_2}$ we have
$$ \|u\|_{L^\ity} \lesssim \|u\|_{\dot{H}^{s_1}}+ \|u\|_{\dot{H}^{s_2}}. $$
\end{proposition}

\section{Modified Bessel functions}
\begin{proposition} \label{FourierModifiedBesselfunctions}
Let $f \in L^p(\R^n)$, $p\in [1,2]$, be a radial function. Then, the Fourier transform $F(f)$ is also a radial function and it satisfies
$$ F(f)(\xi)= c \int_0^\ity g(r) r^{n-1} \tilde{\mathcal{J}}_{\frac{n}{2}-1}(r|\xi|)dr,\quad g(|x|):= f(x), $$
where $\tilde{\mathcal{J}}_\mu(s):=\frac{\mathcal{J}_\mu(s)}{s^\mu}$ is called the modified Bessel function with the Bessel function $\mathcal{J}_\mu(s)$ and a non-negative integer $\mu$.
\end{proposition}

\begin{proposition}\label{PropertiesModifiedBesselfunctions}
The following properties of the modified Bessel function hold:
\begin{enumerate}
\item $s\,d_s\tilde{\mathcal{J}}_\mu(s)= \tilde{\mathcal{J}}_{\mu-1}(s)-2\mu \tilde{\mathcal{J}}_\mu(s)$,
\item $d_s\tilde{\mathcal{J}}_\mu(s)= -s\tilde{\mathcal{J}}_{\mu+1}(s)$,
\item $\tilde{\mathcal{J}}_{-\frac{1}{2}}(s)= \sqrt{\f{2}{\pi}}\cos s$ and $\tilde{\mathcal{J}}_{\frac{1}{2}}(s)= \sqrt{\f{2}{\pi}} \f{\sin s}{s}$,
\item $|\tilde{\mathcal{J}}_\mu(s)| \le Ce^{\pi|\fontshape{n}\selectfont\text{Im}\mu|} \text{ if } s \le 1, $ \\
and $\tilde{\mathcal{J}}_\mu(s)= Cs^{-\frac{1}{2}}\cos \big( s-\frac{\mu}{2}\pi- \frac{\pi}{4} \big) +\mathcal{O}(|s|^{-\frac{3}{2}}) \text{ if } |s|\ge 1$,
\item $\tilde{\mathcal{J}}_{\mu+1}(r|x|)= -\frac{1}{r|x|^2}\partial_r \tilde{\mathcal{J}}_\mu(r|x|)$, $r \ne 0$, $x \ne 0$.
\end{enumerate}
\end{proposition}

\begin{lemma}\label{lemmaB.1}
    We have the following inequality:
    \begin{align*}
        \mathcal{J}_0(x) \geq 1- x^{\alpha} \text{ for all } x \in [0,1],
    \end{align*}
    where $\alpha$ is a suitable constant in $(0,1)$.
\end{lemma}
\begin{proof}
    We consider the function $g(x) = x^{\alpha}+\mathcal{J}_0(x)-1$ with $x \in [0,1]$. Using the second rule with $\mu = 0$ of proposition \ref{PropertiesModifiedBesselfunctions} we obtain
    \begin{align*}
        g'(x) &= \alpha x^{\alpha-1}-\mathcal{J}_1(x)\\
        &= \alpha x^{\alpha-1} -\frac{1}{\pi}\int_0^{\pi}\cos(\tau-x\sin\tau) d\tau\\
        &\geq \alpha x^{\alpha-1} -\frac{1}{\pi}\int_0^{\pi}\cos(\tau-\sin\tau) d\tau.
    \end{align*}
    Now we choose $\alpha = \frac{1}{\pi}\displaystyle\int_0^{\pi}\cos(\tau-\sin\tau) d\tau$ implies $g'(x) \geq 0$ for all $x \in [0,1]$. For this reason, we conclude $g(x) \geq g(0) = 0$ for all $x \in [0,1]$. The proof of the lemma \ref{lemmaB.1} has been completed.
\end{proof}
\section*{Acknowledgments}
This research was partly supported by Vietnam Ministry of Education and Training under grant number B2023-BKA-06.

%=================================================================================={References}

\end{document}